\documentclass[12pt]{amsart}
\usepackage{amssymb}

\newcommand{\eps}{\varepsilon}

\newcommand{\N}{{\mathbb N}}
\newcommand{\C}{{\mathbb C}}
\newcommand{\Z}{{\mathbb Z}}
\newcommand{\R}{{\mathbb R}}
\newcommand{\wt}{\widetilde}
\newcommand{\wand}{wandering domain}
\newcommand{\bwd}{Baker wandering domain}
\newcommand{\pwd}{plane-filling wandering domain}
\newcommand{\tef}{transcendental entire function}
\newcommand{\tmf}{transcendental meromorphic function}
\newcommand{\tmffin}{transcendental meromorphic function with a finite number of poles}

\newcommand{\mf}{meromorphic function}

\newcommand{\sconn}{simply connected}

\theoremstyle{plain}
\newtheorem{theorem}{Theorem}
\newtheorem{corollary}{Corollary}

\newtheorem{lemma}{Lemma}
\theoremstyle{definition}

\theoremstyle{remark}

\theoremstyle{problem}

\theoremstyle{example}
\newtheorem{example}{Example}

\begin{document}


\title[Slow escaping points]{Slow escaping points of meromorphic functions}

\author{P.J. Rippon}
\address{The Open University \\
   Department of Mathematics and Statistics \\
   Walton Hall\\
   Milton Keynes MK7 6AA\\
   UK}
\email{p.j.rippon@open.ac.uk}

\author{G.M. Stallard}
\address{The Open University \\
   Department of Mathematics and Statistics \\
   Walton Hall\\
   Milton Keynes MK7 6AA\\
   UK}
\email{g.m.stallard@open.ac.uk}





\begin{abstract}
We show that for any {\tmf} $f$ there is a point $z$ in the Julia set
of $f$ such that the iterates $f^n(z)$ escape, that is, tend
to~$\infty$, arbitrarily slowly. The proof uses new covering results
for analytic functions. We also introduce several slow escaping sets,
in each of which $f^n(z)$ tends to $\infty$ at a bounded rate, and
establish the connections between these sets and the Julia set of
$f$. To do this, we show that the iterates of $f$ satisfy a strong
distortion estimate in all types of escaping Fatou components except
one, which we call a {\pwd}. We give examples to show how varied the structures
of these slow escaping sets can be.
\end{abstract}

\maketitle

\section{Introduction}
\setcounter{equation}{0} Let $f:\C\to\hat{\C}$ be a {\mf} that is
not rational of degree~1, and denote by
$f^{n},\,n=0,1,2,\ldots\,$, the $n$th iterate of~$f$. The {\it
Fatou set} $F(f)$ is defined to be the set of points $z \in \C$
such that $(f^{n})_{n \in \N}$ is well-defined and forms a normal
family in some neighborhood of $z$.  The complement of $F(f)$ is
called the {\it Julia set} $J(f)$ of $f$. An introduction to the
properties of these sets can be found in~\cite{wB93}.

This paper concerns the {\it escaping set} of $f$, defined as follows:
\[
 I(f) = \{z: f^n(z) \text{ is defined for } n\in \N, f^n(z) \to
 \infty \text{ as } n \to \infty \}.
\]
If $f$ is a polynomial, then $I(f)$ is a neighbourhood of $\infty$
in which iterates tend to $\infty$ at a uniform rate, so
$I(f)\subset F(f)$ and $J(f)=\partial I(f)$; see~\cite{aB}. For a
{\tef}~$f$, the escaping set was first studied  by
Eremenko~\cite{E} who proved that
\begin{equation}\label{E1}
I(f)\cap J(f)\ne \emptyset,
\end{equation}
unlike a polynomial, that
\begin{equation}\label{E2}
J(f)=\partial I(f),
\end{equation}
as for a polynomial, and finally that
\begin{equation}\label{E3}
\text{all components of } \overline{I(f)} \text{ are unbounded}.
\end{equation}
Dom\'inguez~\cite{Pat2} showed that the first two of these properties
are true for any {\tmf}, but the third is not. The set $I(f)$ and the
dynamical behaviour of~$f$ on $I(f)$ are more complicated for a
{\tmf} than for a polynomial. For example, $I(f)$ can have infinitely
many components or it can be connected, and simple examples show that
the set $I(f)$ may or may not meet $F(f)$.

For a {\tef} $f$, the {\it fast escaping set} was introduced by
Bergweiler and Hinkkanen in~\cite{BH99}:
\[
A(f) = \{z: \mbox{there exists } L \in \N \text{ such that }
            |f^{n+L}(z)| > M(R,f^{n}), \text{ for } n \in \N\}.
\]
Here,
\[
M(r,f) = \max_{|z|=r} |f(z)|
\]
and $R$ is any value such that $R > \min_{z \in J(f)}|z|$. The set
$A(f)$ has many properties that make it easier to work with than
$I(f)$; for example, all its components are unbounded~\cite{RS05},
whereas for $I(f)$ this is an open question asked by Eremenko
in~\cite{E}. The set $A(f)$ also meets $J(f)$ and we have
$J(f)=\partial A(f)$; see~\cite{BH99} and~~\cite{RS05}. Note that
$A(f)$ is a subset of
\begin{equation}\label{Z}
Z(f)=\{z\in I(f): \frac{1}{n}\log\log |f^n(z)|\to\infty \text{ as
}n\to\infty\},
\end{equation}
which is the set of points that `zip towards $\infty$';
see~\cite{BH99}. The set $Z(f)$ is defined for all {\tmf}s, it
meets $J(f)$ and we have $J(f)=\partial Z(f)$; see~\cite{RS00}.

It is natural to expect that $Z(f)\ne I(f)$ for every {\tmf}~$f$,
but this has not previously been established.
Here we prove a much stronger result; we show that for all
{\tmf}s~$f$ there are points of $I(f)$ whose iterates tend to
$\infty$ arbitrarily slowly, that is, more slowly than at any
given rate.

\begin{theorem}\label{main1}
Let $f$ be a {\tmf}. Then, given any positive sequence $(a_n)$
such that $a_n\to\infty$ as $n\to\infty$, there exist
\[\zeta \in I(f)\cap J(f)\quad\text{and}\quad N\in\N,\]
such that
\begin{equation}\label{slow}
|f^{n}(\zeta)|\le a_n,\quad\text{for }n\ge N.
\end{equation}
\end{theorem}

Our proof of Theorem~\ref{main1} relies on certain new covering
properties of annuli, which we state and prove in Section~2, 
and on the Ahlfors five islands theorem. We
prove Theorem~\ref{main1} for functions with finitely many poles in
Section~3, and for functions with infinitely many poles in Section~4.
We also indicate how the proof of Theorem~\ref{main1} can be
adapted to construct points such that~(\ref{slow}) holds, and
\[
\liminf_{n\to\infty}|f^n(\zeta)|<\infty
\quad\text{and}\quad\limsup_{n\to\infty}|f^n(\zeta)|=\infty.
\]

Rempe~\cite[Theorem~1.4]{lR06} proved a slow escape result for the
exponential family, by using facts about the structure of the
escaping sets of such functions, which are known to be unions of
curves to infinity called {\it dynamic rays}, or {\it hairs}.
Expressed in terms of functions of the form $f_{\lambda}(z)=\lambda
e^z$, where $\lambda\ne 0$, his result states that for any positive
sequence $(a_n)$ such that $a_n\to\infty$ as $n\to\infty$ and
$a_{n+1}=O(\exp(a_n))$ as $n\to \infty$, there is a point $\zeta\in
J(f_{\lambda})$ -- in fact, an escaping endpoint of a dynamic ray --
and $N\in \N$ such that
\[\frac1C a_n\le |f^n(\zeta)|\le Ca_n,\quad \text{for }n\ge N,\]
where $C=\exp(2+2\pi)$.

In Section~5, we show that a two-sided slow escape result of this
type holds for a wide range of meromorphic functions.

\begin{theorem}\label{main4}
Let $f$ be a {\tmffin} and suppose that there are positive
constants $c$, $d$ and $r_0$ such that $d>1$ and
\begin{equation}\label{minmod}
\text{for all }r\ge r_0\;\text{there exists } \rho\in
(r,dr)\;\text{such that } m(\rho,f)\le c.
\end{equation}
If $(a_n)$ is a positive sequence such that $a_n\to\infty$ as
$n\to\infty$ and $a_{n+1}=O(M(a_n,f))$ as $n\to\infty$, then there
exist $\zeta\in J(f)$ and $C>1$ such that
\begin{equation}\label{twosided}
a_n\le |f^n(\zeta)|\le Ca_n,\quad\text{for } n\in\N.
\end{equation}
\end{theorem}
Here
\[
m(r,f)=\min_{|z|=r} |f(z)|,\quad\text{for }r>0,
\]
denotes the {\it minimum modulus} of~$f$. In particular, Theorem~\ref{main4}
applies whenever~$f$ has a finite number of poles and is bounded on some
path to $\infty$.

It is clear that in Theorem~\ref{main4} some restriction on the
sequence $(a_n)$ is needed, such as $a_{n+1}=O(M(a_n,f))$ as $n\to\infty$. We
prove Theorem~\ref{main4}
in Section~5 and we also point out there why this result requires
some hypothesis such as~(\ref{minmod}) about the minimum modulus
of~$f$. The proof of Theorem~\ref{main4} is considerably simpler
than that of Theorem~\ref{main1}, where extra difficulty arises
from the fact that~$m(r,f)$ may be large for long intervals of
values of~$r$. Note that to extend Theorem~\ref{main4} to more general
meromorphic functions, some replacement for the restriction
$a_{n+1}=O(M(a_n,f))$ as $n\to\infty$ would be needed.

Now we introduce sets of points which escape to $\infty$ at various
bounded rates and investigate to what extent the Eremenko
properties~(\ref{E2}) and~(\ref{E3}) hold for these new sets. First
we define the {\it slow escaping set} of a {\tmf} $f$,
\[
 L(f) = \{z\in I(f):  \limsup_{n\to \infty}\frac1n
 \log|f^n(z)|<\infty
 \},
\]
and the {\it moderately slow escaping set} of $f$,
\[
 M(f) = \{z\in I(f): \limsup_{n\to \infty}\frac1n
 \log\log|f^n(z)|<\infty
 \}.
\]
Evidently we have $L(f)\subset M(f)\subset I(f)\setminus Z(f)$.

Next, for a positive sequence $a=(a_n)$ such that $a_n\to \infty$
as $n\to\infty$ we define
\[
 I^a(f) = \{z\in I(f):  |f^n(z)|=O(a_n) \text{ as } n \to \infty
 \},
\]
where the constant in the $O(.)$ condition depends on the
point~$z$. If $a_n\to \infty$ as $n\to\infty$ and $a_n=O(C^n)$ as
$n\to\infty$, for some $C>1$, then $I^a(f)\subset L(f)$. Note
that each set $L(f)$, $M(f)$ and $I^a(f)$ is non-empty and meets $J(f)$, by
Theorem~\ref{main1}.

To state our results about these sets we need some further
notions. First we describe certain types of Fatou components. For
any component $U$ of $F(f)$ we write $U_n$, for $n\in\N$, to denote the component
of $F(f)$ which contains $f^n(U)$. Then $U$ is a {\it \wand} of
$f$ if the sequence $U_n$ is not periodic or pre-periodic. A {\it
\bwd} of $f$ is a wandering domain $U$ of $f$ such that each component
$U_n$ is bounded, $U_n$ surrounds~0 for $n$ large enough, and
$U_n \to \infty$ as $n\to\infty$. The first example of such a
wandering domain was given by Baker~\cite{iB76}.

We also introduce the notion of a {\it {\pwd}} of $f$, which is a {\wand} $U$ of $f$ such that there is a sequence $(n_j)$ in $\N$ for which
\begin{itemize}
 \item $U_{n_j}$ is bounded and surrounds~0,
 \item $U_{n_j}\to\infty$ as $j\to\infty$.
 \end{itemize}
Any {\bwd} is clearly a {\pwd} and for entire functions all {\pwd}s
are {\bwd}s, by~\cite[Theorem~3.1]{iB84}. An example of a {\tmf} with
a {\pwd} which is not a {\bwd} can be found
in~\cite[Theorem~E]{Pat2}; see also Example~1 of this paper.

Our next result includes a `slow escape' version of Eremenko's
property~(\ref{E2}).
\begin{theorem}\label{main2}
Let $f$ be a {\tmf} and let $a=(a_n)$ be a positive sequence such
that $a_n\to\infty$ as $n\to\infty$. Then
\begin{itemize}
 \item[(a)] $L(f)$ and $M(f)$ are completely
invariant under $f$; \item[(b)] $L(f)$, $M(f)$ and $I^a(f)$ are
each dense in $J(f)$;
 \item[(c)] $J(f)=\partial L(f)=\partial M(f)$;
 \item[(d)]
 \begin{align}
 J(f)&\subset \partial I^a(f)\notag\\
 &\subset J(f)\cup \bigcup\{U: U\;\text{is a
 {\pwd}}\}\notag,
 \end{align}
so if $I^a(f)$ meets no {\pwd}, then $J(f)=\partial I^a(f)$.
\end{itemize}
\end{theorem}
{\it Remarks}\quad 1.\;If $J(f)$ is connected, then the
sets
\[\partial L(f)=\partial M(f)=J(f),\quad\overline{L(f)}\quad\text{and}\quad\overline{M(f)},\]
are also connected, as are $\partial I^a(f)$ and $\overline{I^a(f)}$
provided that $I^a(f)$ meets no {\pwd}, by Theorem~\ref{main2}(c)
and~(d). Note that $J(f)$ is connected if and only if all Fatou
components of $f$ have connected boundaries; see~\cite{K98}.

2.\; In Theorem~\ref{main2}(d), {\pwd}s arise because they are
`exceptional' Fatou components in the following sense: all other
types of Fatou component $U$ have the property that if $\Delta$ is a
compact disc in $U$, then there exist $C>1$ and $n_0\in\N$ such that
\[
 |f^n(z')|\le C|f^n(z)|,\quad \text{for } z,z'\in \Delta,\;n\ge n_0;
\]
see Theorem~\ref{Fatcomp} in Section~6.

3.\; For any {\mf}~$f$ and any sequence $a=(a_n)$ satisfying the
hypotheses of Theorem~\ref{main4}, it is natural to consider subsets
of $I(f)$ of the form
\[
I_a^a(f)=\{z\in I(f): |f^n(z)|\sim a_n\;\text{as }n\to \infty\},
\]
where the constants in the $\sim$ condition depend on the point~$z$,
and to ask if results similar to those in Theorem~\ref{main2} can be
obtained. Results of this type appear to depend on how fast $a_n$
tends to $\infty$ in relation to the speed of escape of points in
$A(f)$; we shall return to this question in a later paper.

Next we describe a family of {\tmf}s which have many dynamical
properties in common with {\tef}s; see~\cite{BRS08}. Let~$D$ be an
unbounded domain in $\C$ whose boundary consists of piecewise smooth
curves, and suppose that $\{z:|z|>r\}\setminus D\neq\emptyset$ for
all $r>0$. Let~$f$ be a complex-valued function whose domain of
definition contains the closure $\overline{D}$ of $D$. Then $D$ is
called a {\em direct tract} of $f$ if the function $f$ is analytic in
$D$ and continuous in $\overline{D}$ and if there exists $R>0$ such
that $|f(z)|=R$ for $z\in\partial D$ and $|f(z)|>R$ for $z\in D$. For
example, any {\tmffin} has at least one direct tract. However, a
{\tmf} with infinitely many poles may or may not have a direct tract.

Recently, the following results about functions with a direct tract
were obtained; see~\cite[Theorem~5.1(a), Theorem~4.1(c) and
Theorem~5.2(a)]{BRS08}.

Let $f$ be a {\tmf} with a direct tract:
\begin{itemize}
 \item if $U$ is a {\bwd} of $f$, then $\overline{U}\subset Z(f)$;
 \item
there is a constant $r_0>0$ such that if $U$ is a component of $F(f)$
which contains a Jordan curve surrounding  $\{z:|z|=r_0\}$, then $U$
is a {\bwd}.
\end{itemize}

In particular, if $f$ is a {\tmf} with a direct tract, then any
{\bwd} of $f$ does not meet $M(f)$, and any {\pwd} of $f$ is a
{\bwd}. Thus we obtain the following corollary of
Theorem~\ref{main2}(d).
\begin{corollary}\label{cor1}
Let $f$ be a {\tmf} with a direct tract and let $a=(a_n)$ be a
positive sequence such that $a_n\to\infty$ as $n\to\infty$. Then
\begin{itemize}
 \item[(a)]
 \begin{align}
 J(f)&\subset \partial I^a(f)\notag\\
 &\subset J(f)\cup \bigcup\{U: U\;\text{is a
{\bwd}}\},\notag
\end{align}
so if $I^a(f)$ meets no {\bwd}, then $J(f)=\partial I^a(f)$;
 \item[(b)] whenever $I^a(f)\subset M(f)$, we have
$J(f)=\partial I^a(f)$.
\end{itemize}
\end{corollary}

We prove Theorem~\ref{main2} in Section~6. We also show there that
Theorem~\ref{main2}(d) and Corollary~\ref{cor1}(a) cannot be improved
to state that we always have $J(f)=\partial I^a(f)$, and that
in Corollary~\ref{cor1}(a) the assumption about
the existence of a direct tract cannot be omitted.

Our final result includes a `slow escape' version of Eremenko's
property~(\ref{E3}) for a {\mf} with a direct tract, and it shows
that a fundamental difference occurs here depending on whether or not
there are {\bwd}s. First recall that if $f$ is a {\tmf} with a direct
tract, then $I(f)$ has at least one unbounded component;
see~\cite[Theorem~1.1]{BRS08}. However, if $f$ has no direct tract,
then $\overline{I(f)}$ may have no unbounded components; for example,
for $f(z)=\frac12 \tan z$ the set $\overline{I(f)}=J(f)$ is totally
disconnected.
\begin{theorem}\label{main3}
Let $f$ be a {\tmf} with a direct tract.
\begin{itemize}
 \item[(a)] Suppose that $a=(a_n)$ is
 a positive sequence such that $a_n\to \infty$ as $n\to\infty$.
 \begin{itemize}
 \item[(i)] If $f$ has no {\bwd}s, then the sets $\partial L(f)$,
 $\partial M(f)$ and $\partial I^a(f)$ all have an unbounded
 component, and so therefore do $\overline{L(f)}$,
 $\overline{M(f)}$ and $\overline{I^a(f)}$.
 \item[(ii)] If $f$ is entire and has no {\bwd}s, then all the
 components of $\partial L(f)$,
 $\partial M(f)$ and $\partial I^a(f)$ are unbounded and so therefore
 are all the components of $\overline{L(f)}$,
 $\overline{M(f)}$ and $\overline{I^a(f)}$.
\end{itemize}
 \item[(b)] If $f$ has a {\bwd}, then all the components of $\overline{M(f)}$ are bounded.
\end{itemize}
\end{theorem}
We prove Theorem~\ref{main3} in Section~7, and in Section~8 we give a
number of examples which show how varied the structures of the sets
$L(f)$, $M(f)$ and $I^a(f)$ can be.

We end this section by making some observations about the possible
relationships between these various subsets of the escaping set and
the components of the Fatou set when~$f$ is a {\tmf}.
\begin{itemize}
 \item Any Fatou component $U$ which meets $I(f)$ must lie in
$I(f)$, and such a Fatou component must be either a {\wand} or a
Baker (or pre-Baker) domain; that is, $U$ maps eventually into a
$p$-cycle of Fatou components in which $f^{np}(z)\to z_0$ as
$n\to\infty$ but $f^p(z_0)$ is not defined.
 \item If $f$ has a Baker domain $U\subset I(f)$, then $U\subset L(f)$,
by~\cite[Theorem~1]{pR06}.
 \item As pointed out earlier in this section, if $f$ has a direct tract
 and $U$ is a {\bwd} of~$f$, then $\overline{U}\subset Z(f)$ and hence
$\overline{U}\cap M(f)=\emptyset$.
 \item There exists a {\tmf}~$f$, with no direct tract, which has a {\bwd} $U$ such that
$\overline{U}\subset L(f)$; see~\cite[discussion after
Theorem~3]{RS08a}.
 \item There exists a {\tef} with a simply
connected {\wand} contained in $A(f)$; see~\cite{wB08}. \item There
exists a {\tef} with a simply connected {\wand} (either bounded or
unbounded) contained in $L(f)$; see Examples~\ref{ex4} and~\ref{ex5}
in Section~8.
\end{itemize}\vspace*{5pt}

\section{Preliminary results}
\setcounter{equation}{0} The construction of the slowly escaping
point in Theorem~\ref{main1} uses the following simple lemma.
\begin{lemma}\label{top}
Let $E_n$, $n\ge 0$, be a sequence of compact sets in $\C$ and
$f:\C\to\hat{\C}$ be a continuous function such that
\begin{equation}\label{contains}
f(E_n)\supset E_{n+1},\quad\text{for } n\ge 0.
\end{equation}
Then there exists $\zeta$ such that $f^n(\zeta)\in E_n$, for $n\ge 0$.

If $f$ is also meromorphic and $E_n\cap J(f)\ne
\emptyset$, for $n\ge 0$, then there exists $\zeta\in J(f)$ such that
$f^n(\zeta)\in E_n$, for $n\ge 0$.
\end{lemma}
\begin{proof}
Let
\[F_n=\{z\in E_0: f(z)\in E_1,\ldots, f^n(z)\in E_n\}.\]
Then, by~(\ref{contains}), $F_n$ is a decreasing sequence of
non-empty compact sets, so $F=\bigcap_{n=0}^{\infty}F_n$ is
non-empty. If $\zeta\in F$, then $f^n(\zeta)\in E_n$, for $n\ge
0$, as required.

The second statement follows by applying the first statement to
the non-empty compact sets $E_n\cap J(f)$, $n\ge 0$, in view of
the complete invariance of $J(f)$.
\end{proof}

In our proof of Theorem~\ref{main1} for functions with a finite
number of poles, we apply Lemma~\ref{top} to sets $E_n$ that are
closed annuli. In order to do this we require two annulus covering
properties. Throughout, we use the following notation, for $z\in\C$
and $0<r<R$:
\begin{itemize}
\item
$A(r,R)=\{z:r<|z|<R\}$,
\item
$B(z,r)=\{w:|w-z|<r\}$.
\end{itemize}

We use a result of Baker and Liverpool~\cite[Lemma~1]{BL74}.

\begin{lemma}\label{BL}
Let $f$ be analytic in the annulus $A(\alpha,\beta)$ and let
$|z_0|=|z|=\sqrt{\alpha\beta}$. If $f$ omits the values~$0$
and~$1$ in $A(\alpha,\beta)$, then
\[|f(z)|\le \exp\left((\log^+|f(z_0)|+C_0)(\exp(\pi^2/\log\gamma)+1)\right),\]
where $C_0$ is a positive absolute constant and
$\gamma=\beta/\alpha$. In particular, if we also have $\gamma\ge 2$,
then
\begin{equation}\label{fbound}
|f(z)|\le (|f(z_0)|+2)^L,
\end{equation}
where $L>2$ is an absolute constant.
\end{lemma}
The estimate~(\ref{fbound}) follows from the fact that
\[\log^+t+C\le 2C\log(t+2),\quad \text{for }t\ge 0,\; C\ge
2.\] We use Lemma~\ref{BL} to prove our first annulus covering
property, which is related to Bohr's lemma;
see~\cite[page~170]{wH}.

\begin{lemma}\label{Bohr}
Let $f$ be a {\tmffin}, let $cL<1/4$, where $c>0$ and $L$ is the
constant in Lemma~\ref{BL}, and let $R_0=R_0(f)>0$ be so large
that $M(r,f)$ is increasing on $[R_0,\infty)$ and
\begin{equation}\label{Mbig1}
M(r,f)>3^{4L},\quad \text{for } r\ge R_0.
\end{equation}
If $r>R_0$ and
\begin{equation}\label{msmall}
\text{there exists } \rho\in (2r,4r)\text{ such that }\;
\log m(\rho,f)\le c\log M(\rho,f),
\end{equation}
then, for any $R$ and $\tilde{R}$ such that
\[2<R\quad\text{and}\quad R^{10}<\tilde{R}<M(r,f)^{1/10},\]
we have
\[
f\left(A(r, 8r)\right)\;\text{covers}\;\;  A(R,R^5)\;\;
\text{or}\;\; A(\tilde{R},\tilde{R}^5).
\]
\end{lemma}
\begin{proof}
Suppose that $r>R_0$. By~(\ref{msmall}), there exists $\rho\in (2r, 4r)$ such that
$m(\rho,f)\le M(\rho,f)^c$. Then $A(r,8r)\supset
A(\tfrac12\rho,2\rho)$.

Now suppose that $f$ omits in $A(\tfrac12\rho, 2\rho)$ two values:
\[
w_1\in A(R,R^5)\quad\text{and}\quad w_2\in A(\tilde{R},\tilde{R}^{5}).
\]
Then $w_1\ne w_2$ and
\[g(z)=\frac{f(z)-w_1}{w_2-w_1}\]
omits in $A(\tfrac12\rho, 2\rho)$ the values 0 and 1, so we can
apply~(\ref{fbound}) to the function $g$.

Take $z_0$ such that $|z_0|=\rho$ and $|f(z_0)|=m(\rho,f)\le
M(\rho,f)^c$. Since $\tilde{R}\ge 2R^5$, we have
\[|g(z_0)|\le \frac{|f(z_0)|+R^5}{\tilde{R}-R^5} \le \frac{|f(z_0)|+R^5}{R^5}\,.\]
Therefore, for $|z|=\rho$, we have, by~(\ref{fbound}),
\begin{align}
|f(z)|&\le |w_1|+(|w_2|+|w_1)|)|g(z)|\notag\\
&\le R^5+2\tilde{R}^5\left(|g(z_0)|+2\right)^L\notag\\
&\le R^5+2\tilde{R}^5\left(\frac{|f(z_0)|+3R^5}{R^5}\right)^L\notag\\
&\le
R^5+2\tilde{R}^5\left(\frac{M(\rho,f)^c+3R^5}{R^5}\right)^L.\notag
\end{align}
Now
\[R<\tilde{R}<M(\rho,f)^{1/10},\quad 0<cL<1/4\quad \text{and}\quad
1+2^{L+1}<3^L<M(\rho,f)^{1/4},\]
by~(\ref{Mbig1}) and the fact that $L>2$. We deduce that, for
$|z|=\rho$,
\begin{align}
|f(z)|&\le M(\rho,f)^{1/2}+2M(\rho,f)^{1/2}\left(M(\rho,f)^c+M(\rho,f)^{1/(4L)}\right)^L\notag\\
&< M(\rho,f)^{3/4}(1+2^{L+1})\notag\\
&< M(\rho,f),\notag
\end{align}
which is a contradiction. Thus $f\left(A(\tfrac12\rho, 2\rho)\right)$
covers at least one of the annuli $A(R,R^5)$ or
$A(\tilde{R},\tilde{R}^{5})$, and so $f\left(A(r,8r)\right)$ does
this also.
\end{proof}

Next, we require a certain Hadamard convexity property.
\begin{lemma}\label{Had1}
Let~$f$ be a {\tmffin}. Then there exists $R_1=R_1(f)>0$ such that
\begin{equation}\label{Had2}
M(r^c,f)\ge M(r,f)^c,\quad\text{for } r\ge R_1,\;c>1.
\end{equation}
\end{lemma}
This result follows from the fact that $\log M(r,f)$ is a convex
function of $\log r$ such that $\log M(r,f)/\log r\to\infty$ as
$r\to\infty$, and hence
\[\frac{rM'(r)}{M(r)}\to\infty\quad\text{as }r\to\infty,\]
where for definiteness we take $M'(r)$ to be the right-derivative;
see~\cite[Lemma~2.2]{RS08} for a proof of Lemma~\ref{Had1} in the
case that $f$ is entire.

We use Lemma~\ref{Had1} to obtain an annulus covering property
in which we assume an opposite type of hypothesis about $m(r,f)$
to that of Lemma~\ref{Bohr}.
\begin{lemma}\label{Harnack}
Let f be a {\tmffin}, let $k>1$, and let $R_2=R_2(f,k)>0$ be so
large that
\begin{itemize}
 \item $M(r,f)$ is increasing on $[R_2,\infty)$,
 \item $M(r,f)>r^k$ and $s=\sqrt{\log r}>\max\{2\pi,4/(k-1)\}$, for $r>R_2$,
 \item the inequality (\ref{Had2}) holds for $r\ge R_2$ and $c>1$.
\end{itemize}
If $r>R_2$ and
\begin{equation}\label{mbig}
m(\rho,f)> 1,\quad\text{for } \rho\in (r^{1+1/s},r^{k-1/s}),
\end{equation}
then
\begin{itemize}
 \item[(a)] we have
 \[
 \log m(\rho,f)\ge \left(1-\frac{2\pi}{s}\right)\log M(\rho,f)>0,
 \quad\text{for } \rho\in [r^{1+2/s},r^{k-2/s}];
 \]
 \item[(b)] we have
 \[A\left(R,R^{k(1-12/s)}\right)\;\text{surrounds}\;\; A(r,r^k),\]
 where $R=M(r^{1+2/s},f)$, and if $A\subset
 A\left(R,R^{k(1-12/s)}\right)$ is any domain such that $\overline{A}$
is homeomorphic to a closed annulus and surrounds~0, then
$A(r^{1+2/s},r^{k-2/s})$ contains a unique component $B$ of
$f^{-1}(A)$, such that
 \begin{itemize}
 \item[$\bullet$] $\overline{B}$ is homeomorphic to a closed
annulus and surrounds~0,
 \item[$\bullet$] $f(\overline{B})=\overline{A}$,
 \item[$\bullet$]
$f$ maps the inner and outer boundary components of $B$ onto the
corresponding boundary components of $A$.
\end{itemize}
\end{itemize}
\end{lemma}
\begin{proof}
By~(\ref{mbig}), the function $u(z)=\log|f(z)|$ is positive
harmonic in $A(r^{1+1/s},r^{k-1/s})$, so $U(t)=u(e^t)$ is positive
harmonic in the strip
\[
S=\{t:\log r+s <\Re(t)<k\log r-s\},
\]
since $s=\sqrt{\log r}$. Now $\log r+2s<k\log r-2s$, since
$s>4/(k-1)$. Thus if $t_1$ and $t_2$ satisfy $\log r
+2s<\Re(t_1)=\Re(t_2)<k\log r-2s$ and $|\Im(t_1)-\Im(t_2)|\le
\pi$, then $B(t_1,s)\subset S$ and $|t_2-t_1|\le \pi<s$. So
\[
\frac{s-\pi}{s+\pi}\le \frac{U(t_2)}{U(t_1)}\le\frac{s+\pi}{s-\pi}\,,
\]
by Harnack's inequality; see~\cite[page~35]{Hay}. Hence, if $z_1$
and $z_2$ satisfy $r^{1+2/s} <|z_1|=|z_2|<r^{k-2/s}$, then
\[
\frac{1-\pi/s}{1+\pi/s}\le \frac{u(z_2)}{u(z_1 )}\le\frac{1+\pi/s}{1-\pi/s}\,.
\]
Since $\pi/s<1/2$, part~(a) then follows.

To prove part~(b), first note that $A\left(R,R^{k(1-12/s)}\right)$
surrounds $A(r,r^k)$ because $R>M(r,f)>r^k$. Next
\begin{equation}\label{inner}
f\left(\{z:|z|=r^{1+2/s}\}\right)\subset\{z:|z|\le
M(r^{1+2/s},f)\}=\{z:|z|\le R\}.
\end{equation}
On the other hand, by part~(a),
\begin{equation}\label{outer}
f\left(\{z:|z|=r^{k-2/s}\}\right)\subset\{z:|z|\ge
M(r^{k-2/s},f)^{1-2\pi/s}\}.
\end{equation}
By~(\ref{Had2}), with
$c=(k-2/s)/(1+2/s)>1$,
\begin{align}
M\left(r^{k-2/s},f\right)^{1-2\pi/s}
&\ge M\left(r^{1+2/s},f\right)^{\frac{k-2/s}{1+2/s}(1-2\pi/s)}\notag\\
&\ge M\left(r^{1+2/s},f\right)^{k(1-12/s)}\notag\\
&=R^{k(1-12/s)}.\notag
\end{align}
Thus, by~(\ref{inner}) and~(\ref{outer}),
\begin{equation}\label{boundary}
f\left(\partial A(r^{1+2/s},r^{k-2/s})\right)\cap
A\left(R,R^{k(1-12/s)}\right)=\emptyset,
\end{equation}
but
$f\left(\partial A(r^{1+2/s},r^{k-2/s})\right)$ does meet both of
the complementary components of $A\left(R,R^{k(1-12/s)}\right)$.
Therefore, since $f$ is analytic in $A(r,r^k)$, we have
\begin{equation}\label{inside}f\left(A(r^{1+2/s},r^{k-2/s})\right)
\supset A\left(R,R^{k(1-12/s)}\right).
\end{equation}

Now let $A\subset A\left(R,R^{k(1-12/s)}\right)$ be a domain such
that $\overline{A}$ is homeomorphic to a closed annulus and
surrounds~0. Then, by~(\ref{boundary}) and~(\ref{inside}), the set
$f^{-1}(A)$ meets $A(r^{1+2/s},r^{k-2/s})$ but does not meet
$\partial A(r^{1+2/s},r^{k-2/s})$. Thus there is at least one
component, $B$ say, of $f^{-1}(A)$ in $A(r^{1+2/s},r^{k-2/s})$ and
$f(\overline{B})=\overline{A}$. By the argument principle, this
component must surround~0 since $f$ is analytic and $f\ne 0$ in
$A(r^{1+2/s},r^{k-2/s})$. Also, $\overline{B}$ has only one bounded
complementary component, namely the one containing~0, because any
other bounded complementary  component must lie in
$A(r^{1+2/s},r^{k-2/s})$ where $f\ne 0$. Hence $\overline{B}$ is
homeomorphic to a closed annulus. The mapping property of the two
boundary components of $B$ follows from the argument principle since
$f:B\to A$ is a proper map. From this we deduce that $B$ is
the unique component of $f^{-1}(A)$ in $A(r^{1+2/s},r^{k-2/s})$ with
these properties. This proves part~(b).
\end{proof}

\section{Functions with finitely many poles}
\setcounter{equation}{0} In this section we prove Theorem~\ref{main1}
for functions with finitely many poles. First we deal with a special
case.
\begin{lemma}\label{special}
Let $f$ be a {\tmffin}. Suppose there is a sequence of continua
$\Gamma_m$, $m\ge 0$, such that
\begin{itemize}
\item[(1)] for each $m\ge 0$, $\Gamma_m$ surrounds~$0$ and has exactly
two complementary components, and $\Gamma_m$ is surrounded by
$\Gamma_{m+1}$; \item[(2)] dist\,$(0,\Gamma_m)\to\infty$ as $m\to \infty$;
\item[(3)] $f(\Gamma_m)= \Gamma_{m+1}$, for $m\ge 0$.
\end{itemize}
Then, given any positive sequence $(a_n)$ such that $a_n\to\infty$
as $n\to\infty$, there exists
\[\zeta \in I(f)\cap J(f)\quad\text{and}\quad N\in\N,\]
such that~(\ref{slow}) holds.
\end{lemma}
\begin{proof}
First note that we can assume that $(a_n)$ is an increasing
sequence. Also, by renumbering if necessary, we can assume that
$f$ has no poles on~$\Gamma_0$ or outside~$\Gamma_0$.

By~(1) we can define $B_m$, for $m\ge 0$, to be the
union of $\Gamma_m$ and $\Gamma_{m+1}$ and those points that are both
outside $\Gamma_m$ and inside $\Gamma_{m+1}$. Then $B_m$ is a continuum that
surrounds~$0$ and $\partial B_m$ is a subset of $\Gamma_m \cup
\Gamma_{m+1}$. Thus by (3) and the fact that~$f$ is analytic in a
neighbourhood of $B_m$ we have, for each $m\ge 0$, exactly one of
the following possibilities:
\begin{equation}\label{onto}
f(B_m)=B_{m+1},
\end{equation}
\begin{equation}\label{notonto}
f(B_m)= G_{m+1}\cup B_{m+1}\,,
\end{equation}
where $G_{m+1}$ is the complementary component of $B_{m+1}$ that
contains~$0$. Since $J(f)$ is unbounded, we deduce that
\begin{equation}\label{BJul}
B_m\cap J(f)\ne \emptyset,\quad \text{for all } m\ge 0.
\end{equation}

If~(\ref{onto}) holds for all $m\ge M$, say, then (by the
hypotheses~(1) and~(2)) the outside of $B_M$ is contained in the
Fatou set of $f$, which is impossible. Thus there is a strictly
increasing sequence $m(j)\in\N$ such that~(\ref{notonto}) holds for
all $m=m(j)$, $j\in\N$. Hence we have the covering properties
\begin{equation}\label{next1}
f(B_{m})\supset B_{m+1},\quad\text{for }m\ge 0,
\end{equation}
and
\begin{equation}\label{same1}
f(B_{m(j)})\supset B_{m(j)},\quad\text{for } j\in\N.
\end{equation}
By~(\ref{same1}), for any $d\in \N$, we have
\[f^{d}(B_{m(j)})\supset B_{m(j)},\quad\text{for }j\in\N.\]
The idea now is to choose a point $\zeta\in B_0$ which has an
orbit that visits each of the compact sets $B_m$, $m\ge 0$, in
order of increasing $m$, except that the orbit remains in each
$B_{m(j)}$, $j\in\N$, for $d(j)$ steps. To arrange this, we
introduce a sequence $p(j)$ of the form
\[p(j)=d(1)+\cdots +d(j),\quad j\in \N,\]
where $d(j)\in \N$. The sequence $(d(j))$ will be chosen later to
give the desired rate of escape of $f^n(\zeta)$. We also put $m(0)=0$
and $p(0)=0$.

If we define
\[
E_n=\left\{
 \begin{array}{ll}
 B_{n-p(j-1)}, &\text{for }  m(j-1)+p(j-1)\le n< m(j)+p(j-1),\;j\in \N, \\
 B_{m(j)}, &\text{for }  m(j)+p(j-1)\le n< m(j)+p(j),\;j\in \N,
 \end{array}
   \right.
\]
then it follows from~(\ref{next1}) and~(\ref{same1})
that~(\ref{contains}) holds. Thus, by Lemma~\ref{top}
and~(\ref{BJul}), there exists a point $\zeta\in E_0\cap J(f)=B_0\cap
J(f)$ such that, for $j\in \N$,
\[f^n(\zeta)\in B_{n-p(j-1)},\quad \text{for }m(j-1)+p(j-1)\le n< m(j)+p(j-1),\]
and
\[f^n(\zeta)\in B_{m(j)},\quad \text{for }m(j)+p(j-1)\le n< m(j)+p(j).\]
Clearly $\zeta\in I(f)\cap J(f)$ for all possible choices of
$p(j)$.

To complete the proof, we choose a subsequence $(a_{n(j)})$ of
$(a_n)$ such that
\[B_{m(j)}\subset B(0, a_{n(j)}),\]
and then choose $d(j)$ so large that $m(j-1)+p(j-1) \ge n(j)$, for
$j\ge 2$. Then, for $j\ge 2$ and $m(j-1)+p(j-1)\le n< m(j)+p(j)$, the
point $f^n(\zeta)$ lies inside the outer boundary of $B_{m(j)}$, so
\[|f^n(\zeta)|\le a_{n(j)}\le a_{m(j-1)+p(j-1)}\le a_n,\]
since $(a_n)$ is increasing. This proves~(\ref{slow}).
\end{proof}

We now prove Theorem~\ref{main1} for {\it any} {\tmffin}. Once again
we can assume that $(a_n)$ is an increasing sequence such that
$a_n\to\infty$ as $n\to\infty$. 

First take $r_0$ so large that $r_0\ge \max\{R_0,R_1, R_2\}$, where
$R_0=R_0(f)$, $R_1=R_1(f)$ and $R_2=R_2(f,k)$, $4\le k\le 5$, are the
constants appearing in Lemmas~\ref{Bohr},~\ref{Had1}
and~\ref{Harnack}, and also
\begin{equation}\label{r0big}
r_0\ge \exp(120^2)
\end{equation}
and
\begin{equation}\label{Mbig}
\frac{\log M(r,f)}{\log r}\ge 1000,\quad\text{for }r\ge r_0.
\end{equation}

Consider the annulus $A_0=A(r_0,r_0^{k_0})$, where $k_0=5$. The
first step of the proof is to use Lemma~\ref{Bohr} and
Lemma~\ref{Harnack} to choose a sequence of annuli of the form
\[A_m=A(r_m,r_m^{k_m}),\quad m\ge 0,\]
such that, for $m\in \N$,
\begin{equation}\label{cond1}
f(A_{m-1})\supset A_{m},
\end{equation}
\begin{equation}\label{cond2}
r_m>r_{m-1}^{10}\quad\text{and}\quad k_m\ge k_{m-1}(1-12/s_{m-1}),
\end{equation}
where $s_m=\sqrt{\log r_m}$, and
\begin{equation}\label{cond3}
4\le k_m\le 5.
\end{equation}
In particular, note that $A_m$ surrounds $A_{m-1}$
and $r_m\to\infty$ as $m\to\infty$.

Suppose that the annuli $A_0, A_1, \ldots, A_{m-1}$, $m\in\N$, have been chosen so
that they satisfy the above conditions. To choose $A_m$ we consider
two cases.

{\it Case 1}\quad Suppose first that
\begin{equation}\label{rho}
\text{there exists } \rho\in (3r_{m-1},\tfrac38 r_{m-1}^{k_{m-1}})
\text{ such that }\;m(\rho,f)\le 1.
\end{equation}
Then
\[\rho\in (\tfrac23\rho,\tfrac43\rho)\subset(\tfrac13\rho,\tfrac83\rho)
\subset (r_{m-1},r_{m-1}^{k_{m-1}}).\]
Since $M(\rho,f)\ge M(r_0,f)\ge 1$, by~(\ref{r0big}) and~(\ref{Mbig}), it follows from~(\ref{rho}) that we can 
apply Lemma~\ref{Bohr} with $r=\tfrac13 \rho$. We choose~$R$
and~$\tilde{R}$ such that
\begin{equation}\label{lemma2}
r_{m-1}^{10}<R\quad\text{and}\quad
R^{10}<\tilde{R}<M(r_{m-1},f)^{1/10};
\end{equation}
this is possible by~(\ref{Mbig}). With $r=\tfrac13\rho$, we deduce
from~(\ref{rho}),~(\ref{lemma2}) and Lemma~\ref{Bohr} that
\[
f\left(A(r_{m-1}, r_{m-1}^{k_{m-1}})\right)\;\text{covers}\;\;
A(R,R^5)\;\;\text{or}\;\; A(\tilde{R},\tilde{R}^5).
\]
Hence we can choose $r_m=R$ or $r_m=\tilde{R}$ and $k_m=5$ to
ensure that~(\ref{cond1}),~(\ref{cond2}) and~(\ref{cond3}) also hold for $A_m$.

{\it Case 2}\quad On the other hand, suppose that~(\ref{rho}) is
false; that is,
\[
m(\rho,f)> 1,\quad\text{for all } \rho\in (3r_{m-1},\tfrac38
r_{m-1}^{k_{m-1}}).
\]
Then
\[
m(\rho,f)> 1,\quad\text{for all } \rho\in
(r_{m-1}^{1+1/s_{m-1}},r_{m-1}^{k_{m-1}-1/s_{m-1}}),
\]
because
\[(1/s_{m-1})\log r_{m-1}=\sqrt{\log r_{m-1}}>\log 3,\]
by~(\ref{r0big}). Thus, by Lemma~\ref{Harnack} and~(\ref{Mbig}),
\[
f\left(A(r_{m-1}, r_{m-1}^{k_{m-1}})\right)\supset A\left(R,
R^{k_{m-1}(1-12/s_{m-1})}\right),
\]
where $R=M(r_{m-1}^{1+2/s_{m-1}},f)>r_{m-1}^{10}$. Thus we can
choose $r_m=R$ and $k_m=k_{m-1}(1-12/s_{m-1})$ to ensure
that~(\ref{cond1}) and~(\ref{cond2}) hold for $A_m$. To see
that~(\ref{cond3}) also holds for $A_m$, note that, for $1\le j\le
m-1$ we have, by~(\ref{cond2}),
\[k_j\ge k_{j-1}(1-12/s_{j-1})\]
and
\[\frac{s_{j}}{s_{j-1}}=\sqrt{\frac{\log r_{j}}{\log r_{j-1}}}
\ge 3.\]
Also, $k_0=5$ and $s_0\ge 120$, by~(\ref{r0big}).
Thus
\[\frac{12}{s_j}\le \frac{1}{10}\,\left(\frac{1}{3}\right)^{j},\quad\text{for }0\le j\le m-1,\]
so
\[k_m\ge 5\prod_{j=0}^{m-1}\left(1-\frac{1}{10}\,\left(\frac{1}{3}\right)^j\right)\ge 4,\]
as required.

We have now shown that it is possible to construct a sequence of
annuli $A_m$ satisfying conditions~(\ref{cond1}),~(\ref{cond2})
and~(\ref{cond3}). Next we use these annuli to obtain a point
$\zeta\in I(f)\cap J(f)$ satisfying~(\ref{slow}).

First suppose that in this process of choosing the annuli we were in
Case~2 for all $m\ge M$, say. Without loss of generality we can
assume that $M=0$. So, for each $m\in\N$, the annulus~$A_m$ was
obtained by applying Lemma~\ref{Harnack} to $A_{m-1}$. We can also
deduce from Lemma~\ref{Harnack} that if $A\subset
A\left(r_m,r_m^{k_m}\right)$ is any domain such that 
$\overline{A}$ is homeomorphic to a closed annulus
which surrounds~0, then $A_{m-1}$ contains a unique component $B$ of
$f^{-1}(A)$ such that $\overline{B}$ is homeomorphic to a closed
annulus and surrounds~0, $f(\overline{B})=\overline{A}$, and $f$ maps
the inner and outer boundary components of $B$ onto the corresponding
boundary components of $A$.

Therefore, for each $n\ge m\ge 0$ there is a unique set $\Gamma_{m,n}$,
homeomorphic to a closed annulus, contained in $\overline{A_m}$
such that
\begin{itemize}
 \item
$f^{n-m}(\Gamma_{m,n})=\Gamma_{n,n}=\overline{A_n}$, for $n\ge m\ge 0$;
 \item
$\Gamma_{m,n}$ surrounds~0, for $n\ge m\ge 0$;
 \item for $n>m\ge 0$, we have $f(\Gamma_{m,n})=\Gamma_{m+1,n}$,
 and $f$ maps the inner and outer boundary components of $\Gamma_{m,n}$ onto the corresponding
boundary components of $\Gamma_{m+1,n}$;
 \item $\Gamma_{m,n+1}\subset \Gamma_{m,n}$, for $n\ge m\ge 0$.
\end{itemize}
Now let $\Gamma_m=\bigcap_{n\ge m}\Gamma_{m,n}$, for $m\ge 0$. Then each
$\Gamma_m$, $m\ge 0$, is a continuum which surrounds~$0$ and has two
complementary components, and also $\Gamma_{m+1}$ surrounds $\Gamma_m$.
Moreover, for each $m\ge 0$, we have $f(\Gamma_m)=\Gamma_{m+1}$, since
$f(\Gamma_{m,n})=\Gamma_{m+1,n}$, for all $n>m$. Thus the sequence of
continua $\Gamma_m$ has properties~(1),~(2) and~(3) of
Lemma~\ref{special}, so there exists a point $\zeta\in I(f)\cap
J(f)$ satisfying~(\ref{slow}).

The alternative is that in the process of choosing the annuli
$A_m$ we were in Case~1 infinitely often; that is, we obtained
$A_m$ by applying Lemma~\ref{Bohr} to $A_{m-1}$ for infinitely
many~$m$. When we apply Lemma~\ref{Bohr} to $A_{m-1}$, for $m\ge
2$, there exists~$\rho$ such that~(\ref{rho}) holds and for this
value of $\rho$ we have
\[r_{m-2}^{10}<r_{m-1}<M(r_{m-1},f)^{1/10}<M(\tfrac13\rho,f)^{1/10},\]
by~(\ref{cond2}),~(\ref{Mbig}) and the fact that $M(r,f)$ is
increasing on $[r_0,\infty)$. Thus, by applying Lemma~\ref{Bohr}
with $r=\tfrac13\rho$ again, but $R=r_{m-2}$, and
$\tilde{R}=r_{m-1}$, we obtain
\[
f\left(A_{m-1}\right)\;\text{covers}\;\; A_{m-2}\;\;\text{or}\;\;
A_{m-1},\;\;\text{so}\;\; f^{2}\left(A_{m-1}\right)\supset
A_{m-1}.
\]
Hence in this situation we have
\begin{equation}\label{next2}
f(A_{m-1})\supset A_m,\quad\text{for }m\in \N,
\end{equation}
and there is a strictly increasing sequence of positive integers
$m(j)$, $j\ge 0$, such that
\begin{equation}\label{same2}
f\left(A_{m(j)}\right)\;\text{covers}\;\; A_{m(j)-1}\;\;\text{or}\;\;
A_{m(j)},\;\;\text{so}\;\;f^{2}\left(A_{m(j)}\right)\supset A_{m(j)},\;\;\text{for }j\in \N.
\end{equation}

Now we choose a point $\zeta\in \overline{A_0}$ which has an orbit
that visits each of the annuli $\overline{A_m}$, $m\ge 0$, in order
of increasing $m$, except that after entering $\overline{A_{m(j)}}$,
$j\in\N$, for the first time the orbit remains in
$\overline{A_{m(j)-1}}\cup \overline{A_{m(j)}}$ for $d(j)$ steps,
ending in~$A_{m(j)}$. The fact that such a point~$\zeta$ exists can
be deduced from~(\ref{next2}) and~(\ref{same2}) by using
Lemma~\ref{top} in the same way that the existence of the point
$\zeta$ is deduced from~(\ref{next1}) and~(\ref{same1}) in
Lemma~\ref{special}. The only difference here is that the positive
integers $d(j)$ must be even because of~(\ref{same2}). Clearly
$\zeta\in I(f)$ once again and by choosing the integers $d(j)$
appropriately we can ensure that $\zeta$ satisfies~(\ref{slow}), as
in the proof of Lemma~\ref{special}.

To complete the proof of Theorem~\ref{main1} we show that we can
take $\zeta\in J(f)$. We do this, using the second statement in
Lemma~\ref{top}, by proving that each of the closed annuli
$\overline{A_m}$ must meet $J(f)$.

Suppose that $\overline{A_{m(j)}}\subset F(f)$ for some $j\in\N$.
Then $\overline{A_{m(j)}}\subset I(f)$ by normality, since
$\overline{A_{m(j)}}\cap I(f)\ne\emptyset$. However,
by~(\ref{same2}) and Lemma~\ref{top}, for each $j\in \N$ there
exists a point $z_j\in \overline{A_{m(j)}}$ such that the forward
orbit of $z_j$ remains in $\overline{A_{m(j)-1}}\cup
\overline{A_{m(j)}}$, so $z_j\notin I(f)$, a contradiction.
Therefore our supposition must be false. Hence the sets
$\overline{A_m}$ meet $J(f)$ for arbitrarily large $m$ and hence
for all~$m$ by~(\ref{next2}), as required. This completes the
proof of Theorem~\ref{main1} for functions with finitely many
poles.

{\it Remark}\quad  Whenever we apply Lemma~\ref{Bohr}
or~(\ref{notonto}) we have the option to choose the covered annulus
to be either large or within a uniformly bounded distance of~0. If we
do this, then the corresponding points of the orbit of $\zeta$ have
the same property. Hence we can ensure that~(\ref{slow}) holds, and
\[
\liminf_{n\to\infty}|f^n(\zeta)|<\infty
\quad\text{and}\quad\limsup_{n\to\infty}|f^n(\zeta)|=\infty.
\]

\section{Functions with infinitely many poles}
\setcounter{equation}{0}

For functions with infinitely many poles, we prove
Theorem~\ref{main1} using the following version of Ahlfors' five
islands theorem; see~\cite[Corollary to Theorem~VI.8]{T}.

\begin{lemma}\label{islands}
If $f$ is a {\tmf} and $D_i$, $1\le i\le 5$, are {\sconn} domains
bounded by Jordan curves such that the $\overline{D_i}$ are
disjoint, then for each $R>0$ there are infinitely many domains in
$\C\setminus B(0,R)$ each of which is mapped by $f$ univalently
onto one of the $D_i$.
\end{lemma}

Suppose that $f$ is a {\tmf} with infinitely many poles. We have
the following properties:
\begin{itemize}
\item[(1)] the image under $f$ of any open disc around a pole of
$f$ contains a neighbourhood of $\infty$; \item[(2)] if $D_0, D_1,
D_2, D_3, D_4$ are open discs with disjoint closures and $R>0$,
then, by Lemma~\ref{islands}, there exists a Jordan domain $D$ in
$\C\setminus B(0,R)$ such that $f$ maps $D$ univalently onto
$D_i$, for some $i\in \{0,1,\ldots, 4\}$.
\end{itemize}

We now obtain a sequence of Jordan domains tending to $\infty$ with
certain covering properties. First, we use property~(1) and the fact
that~$f$ has infinitely many poles, to obtain open discs
\[D_m=B(z_m,r_m), \quad m\ge 0,\]
such that
\begin{equation}\label{infD}
{\rm dist}\,(0,D_m)\to\infty\;\text{ as }m\to\infty,
\end{equation}
\begin{equation}\label{poles}
z_{m} \;\;\text{is a pole of  }f,\quad\text{for }m\ge 0,
\end{equation}
and
\begin{equation}\label{contains5}
f(D_{m})\supset D_{m+1},\quad\text{for }m\ge 0.
\end{equation}
Next, we use property~(2) to obtain Jordan domains $V_{j}$, $j\ge 0$,
such that
\begin{equation}\label{infV}
{\rm dist}\,(0,V_j)\to\infty\;\text{ as }j\to\infty,
\end{equation}
\begin{equation}\label{contains3} f(D_{5j+4})\supset
V_{j},\quad\text{for }j\ge 0,
\end{equation}
and
\begin{equation}\label{contains2}
f(V_j)\supset D_{m(j)},\quad\text{for }j\ge 0\;\text{and some
}m(j)\in\{5j,5j+1,\ldots,5j+4\}.
\end{equation}
By~(\ref{contains5}),~(\ref{contains3}) and~(\ref{contains2}), we
have
\begin{equation}\label{cycle}
f^{60}(D_{5j+4})\supset D_{5j+4},\quad\text{for }j\ge 0,
\end{equation}
since~$60$ is the least common multiple of $2,3,4,5$ and $6$.

The idea now is to choose a point $\zeta\in \overline{D_0}$ which has
an orbit that visits each of the sets $\overline{D_m}$, $m\ge 0$, in
order of increasing $m$, except that after entering $D_{5j+4}$, $j\ge
0$, for the first time the orbit remains in
$\overline{D_{5j}}\cup\cdots \cup \overline{D_{5j+4}}\cup
\overline{V_j}$ for $d(j)$ steps ending in $\overline{D_{5j+4}}$.

To arrange this, we introduce a sequence $p(j)$ of the form
\[p(j)=d(0)+\cdots +d(j),\quad j\ge 0,\]
where each $d(j)\in 60\N$, and also put $p(-1)=0$. Then, for $j\ge
0$, define
\begin{equation}\label{Edefmero}
E_n=\overline{D_{n-p(j-1)}},\quad\text{for } 5j+p(j-1)\le
n<5j+5+p(j-1),
\end{equation}
and define $E_n$, for $5j+5+p(j-1)\le n< 5j+5+p(j)$, to be $d(j)$
closed sets, each belonging to
$\{\overline{D_{5j}},\overline{D_{5j+1}},\ldots,\overline{D_{5j+4}},\overline{V_j}\}$,
which are arranged in the order defined by the covering
properties~(\ref{contains2}) and~(\ref{cycle}), starting with
$\overline{V_j}$ and ending with $\overline{D_{5j+4}}$.
Then~(\ref{contains}) holds, by~(\ref{contains5}) and~(\ref{cycle}),
so we can use Lemma~\ref{top} to choose the required point $\zeta\in
\overline{D_0}$ such that, for $j\ge 0$,
\[f^n(\zeta)\in \overline{D_{n-p(j-1)}},\quad\text{for }5j+p(j-1)\le n<5j+5+p(j-1),\]
and
\[
f^n(\zeta)\in \overline{D_{5j}}\cup\cdots \cup
\overline{D_{5j+4}}\cup \overline{V_j},\quad \text{for
}5j+5+p(j-1)\le n< 5j+5+p(j).
\]
Clearly $\zeta\in I(f)$ for all possible choices of $p(j)$
by~(\ref{infD}) and~(\ref{infV}). Also, given any positive increasing
sequence $(a_n)$ such that $a_n\to\infty$ as $n\to\infty$, we can
choose $d(j)$ appropriately to ensure that $\zeta$
satisfies~(\ref{slow}), as in the proof of Lemma~\ref{special}.
Finally, note that all poles and their pre-images are in $J(f)$, so
we deduce from the second statement of Lemma~\ref{top} that $\zeta$
can be taken to lie in $J(f)$, as required. This completes the proof
of Theorem~\ref{main1}.

{\it Remark}\quad  Whenever we apply Lemma~\ref{islands} we have the
option to choose the covered disc to be either large or within a
uniformly bounded distance of~0. If we do this, then the
corresponding points of the orbit of $\zeta$ have the same property.
Hence we can ensure that~(\ref{slow}) holds, and
\[
\liminf_{n\to\infty}|f^n(\zeta)|<\infty
\quad\text{and}\quad\limsup_{n\to\infty}|f^n(\zeta)|=\infty.
\]

\section{Proof of Theorem~2}
\setcounter{equation}{0} To prove Theorem~\ref{main4} we use
another annulus covering property related to Bohr's lemma.

\begin{lemma}\label{Bohr1}
Let $f$ be a {\tmffin}, let $c$ and $K$ be positive constants,
and let $ \alpha, \beta,\gamma, \alpha', \beta', \alpha'', \beta''$
and $C$ be constants such that
\begin{equation}\label{alpha}
1<\alpha<\beta,\quad 1<\alpha'<\beta'<\alpha''<\beta''\le C
\quad\text{and}\quad\alpha''/\beta'\ge \gamma=\beta/\alpha.
\end{equation}
There exists $R_0=R_0(f,\gamma,C,c,K)>0$ such that if $r>R_0$,
\begin{equation}\label{msmall1}
1\le R\le KM(r,f)\;\;\text{and}\;\; m(\sqrt{\alpha\beta}\, r,f)\le c,
 \end{equation}
then
\[
f\left(A(\alpha r, \beta r)\right)\;\text{covers}\;\;
A(\alpha'R,\beta'R)\;\; \text{or}\;\; A(\alpha''R,\beta''R).
\]
\end{lemma}
\begin{proof}
Assume that~(\ref{msmall1}) holds for some $r>0$. Then there exists $z_0$ such that
$|z_0|=\sqrt{\alpha\beta}\, r$ and $|f(z_0)|\le c$. Suppose
that~$f$ omits in $A(\alpha r, \beta r)$ two values:
\[
w_1\in A(\alpha'R,\beta'R)\quad\text{and}\quad w_2\in
A(\alpha''R,\beta''R).
\]
Then $w_1\ne w_2$ and
\[g(z)=\frac{f(z)-w_1}{w_2-w_1}\]
omits in $A(\alpha r, \beta r)$ the values 0 and 1, so we can
apply~Lemma~\ref{BL} to the function~$g$. Now
\[|g(z_0)|\le \frac{c+\beta'R}{\alpha''R-\beta'R}\le \frac{c+1}{\alpha''/\beta'-1}\,.\]
Thus, by Lemma~\ref{BL} and~(\ref{alpha}), for $|z|=\sqrt{\alpha\beta}\,r$,
\[
 |g(z)|\le \exp\left(\left(\log^+\left(\frac{c+1}{\gamma-1}\right)+C_0\right)
 \left(\exp\left(\frac{\pi^2}{\log\gamma}\right)+1\right)\right)= D,
\]
where $C_0$ is a positive absolute constant and the positive constant
$D$ depends on~$\gamma$ and~$c$. Therefore, for
$|z|=\sqrt{\alpha\beta}\,r$, we have
\begin{align}
|f(z)|&\le |w_1|+(|w_2|+|w_1)|)|g(z)|\notag\\
&\le \beta'R+(\beta''R+\beta'R)D\notag\\
&\le (1+2D)C KM(r,f),\notag
\end{align}
by~(\ref{alpha}) and~(\ref{msmall1}).

Now, for any $k>1$ we have
\[\frac{M(kr,f)}{M(r,f)}\to\infty\;\;\text{as }r\to\infty,\]
since $\log M(r,f)$ is a convex function of $\log r$ such that $\log M(r,f)/\log r\to\infty$ 
as $r\to\infty$. Therefore, there exists $R_0=R_0(f,\gamma,C,c,K)$ such that if
$r>R_0$, then
\[
|f(z)|<M(\sqrt{\alpha\beta}\,r,f),\quad\text{for }
|z|=\sqrt{\alpha\beta}\,r,
\]
which is a contradiction. Thus we deduce that if $r>R_0$, then
$f\left(A(\alpha r, \beta r)\right)$ covers at least one of the
annuli $A(\alpha'R,\beta'R)$ or $A(\alpha''R,\beta''R)$, as required.
\end{proof}
We also require the following result due to Zheng~\cite{jhZ06}.
\begin{lemma}\label{Zheng}
Let $f$ be a {\tmf} with at most finitely many poles. If $f$ has a
{\bwd} $U$, then for a multiply connected domain $A$ in $U$ such that
each $f^n(A)$, $n\in\N$, contains a closed curve which is not
null-homotopic in $U_n$, there exist annuli $A_n =\{z:r_n <
|z|<R_n\}$, $n\in\N$, and $n_0\in\N$ such that
\[A_n\subset f^n(A), \quad\text{for } n > n_0,\]
$dist(0,A_n)\to\infty$ as $n\to\infty$ and $R_n/r_n\to\infty$ as
$n\to\infty$.
\end{lemma}
\begin{proof}[Proof of Theorem~\ref{main4}]
In the statement of Theorem~\ref{main4}, $f$ is a {\tmffin}, and $c$,
$d$ and $r_0$ are positive constants such that $d>1$ and
\begin{equation}\label{minmod1}
\text{for all }r\ge r_0\;\text{there exists } \rho\in
(r,dr)\;\text{such that } m(\rho,f)\le c.
\end{equation}
Also, $(a_n)$ is a positive sequence such that $a_n\to\infty$ as $n\to\infty$ and
$a_{n+1}=O(M(a_n,f))$ as $n\to\infty$.

We take a positive constant $K$ and $N\in\N$ so large that
\begin{equation}\label{anests}
1\le a_{n+1}\le KM(a_n,f)\quad\text{and}\quad a_n\ge \max\{r_0,R_0\},\quad \text{for } n\ge N,
\end{equation}
where $R_0=R_0(f,d,d^6,c,K)$ is the constant in Lemma~\ref{Bohr1}. Then,
by~(\ref{minmod1}), there exist sequences
$(\rho'_n)$ and $(\rho''_n)$ such that, for $n\ge N$,
\begin{equation}\label{rhodash}
\rho'_n\in (da_n,d^2a_n)\quad\text{and} \quad m(\rho'_n,f)\le c,
\end{equation}
and
\begin{equation}\label{rhoddash}
\rho''_n\in (d^4a_n,d^5a_n)\quad\text{and} \quad m(\rho''_n,f)\le c.
\end{equation}
Now define, for $n\ge N$,
\[
A'_n=A(d^{-1/2}\rho'_n, d^{1/2} \rho'_n)\quad\text{and}\quad
A''_n=A(d^{-1/2}\rho''_n, d^{1/2} \rho''_n).
\]
Then
\begin{equation}\label{constants}
a_n<d^{-1/2}\rho'_n <d^{1/2} \rho'_n < d^{-1/2}\rho''_n <d^{1/2}
\rho''_n < d^{\,6}a_n.
\end{equation}
We now apply Lemma~\ref{Bohr1} with
\begin{equation}\label{definer}
r=a_n,\; \alpha r=d^{-1/2}\rho'_n, \;\beta r=d^{1/2} \rho'_n,
\end{equation}
and
\[
R=a_{n+1},\;\alpha' R=d^{-1/2}\rho'_{n+1}, \;\beta' R=d^{1/2}
\rho'_{n+1},\;\alpha'' R=d^{-1/2}\rho''_{n+1}, \;\beta'' R=d^{1/2}
\rho''_{n+1}.
\]
Then $\alpha''/\beta'=d^{-1}\rho''_{n+1}/\rho'_{n+1}\ge d$, 
by~(\ref{rhodash}) and~(\ref{rhoddash}), and $\beta/\alpha=d$, by~(\ref{definer}). We deduce from Lemma~\ref{Bohr1} and~(\ref{anests}) that for $n\ge N$ we have
\[
f(A'_n)\quad\text{covers} \quad
A'_{n+1}\;\;\text{or}\;\;A''_{n+1};
\]
similarly,
\[
f(A''_n)\quad\text{covers} \quad
A'_{n+1}\;\;\text{or}\;\;A''_{n+1}.
\]
Therefore we can choose, for $n\ge N$, the compact set $E_n$ to be
either $\overline{A'_n}$ or $\overline{A''_n}$ in such a way that
\begin{equation}\label{supset}
f(E_n)\supset E_{n+1},\quad\text{for } n\ge N.
\end{equation}
Then, by Lemma~\ref{top} and~(\ref{constants}), there exists
$\zeta_N\in E_N$ such that
\[
f^{n-N}(\zeta_N)\in E_n\subset
\overline{A(a_n,Ca_n)},\quad\text{for } n\ge N,
\]
where $C=d^{\,6}$.

Without loss of generality, we can assume that $\zeta_N$ is not a
Fatou-exceptional value, so by applying Picard's theorem a finite
number of times we can choose~$\zeta$ such that
$f^N(\zeta)=\zeta_N$ and $|f^n(\zeta)|\ge a_n$, for $n=1,\ldots,
N-1$. Thus, for this~$\zeta$ we have (possibly with a larger
constant $C$)
\[
a_n\le |f^n(\zeta)|\le Ca_n,\quad\text{for } n\in\N.
\]

To ensure that we also have $\zeta\in J(f)$, we observe that $f$
cannot have a {\bwd} $U$. For this would imply,
by Lemma~\ref{Zheng}, the existence of a sequence of
annuli $A(r_n,R_n)$, where $r_n<R_n$, such that $r_n\to\infty$ as $n\to\infty$ and
$R_n/r_n\to\infty$ as $n\to\infty$, and for $n$ large enough
\[
\overline{A(r_n,R_n)}\subset f^{n}(U).
\]
Thus $m(r,f)\to \infty$ as $r\to\infty$ through
\[
 \bigcup_{n=1}^{\infty} (r_n,R_n),
\]
contrary to the hypothesis~(\ref{minmod}).
By~\cite[Theorem~1]{jhZ02} or~\cite[Theorem~5]{RS05a}, it follows
that $J(f)$ has an unbounded component. Thus, by~(\ref{supset}), we can choose
$\zeta_N\in E_N\cap J(f)$, as required. This completes the proof
of Theorem~\ref{main4}.
\end{proof}
Finally in this section we point out why we cannot expect
Theorem~\ref{main4} to hold without some hypothesis such
as~(\ref{minmod}) about the minimum modulus of~$f$. Suppose that
there exists a sequence of annuli $A(r_n,R_n)$, where $0<r_n<R_n$,
such that $r_n\to\infty$ as $n\to\infty$, $R_n/r_n\to\infty$ as
$n\to\infty$ and
\[
m(r,f)>1,\quad\text{for } r_n<r<R_n, \;n\in \N.
\]
Then by an argument similar to that in the proof of
Lemma~\ref{Harnack} we deduce that
\[
m(r,f)>M(r,f)^c,\quad\text{for } 2r_n<r<\tfrac12 R_n, \;n\ge N,
\]
for some constant $0<c<1$ and some $N\in\N$. Since
$M(r,f)^c/r\to\infty$ as $r\to\infty$, it is not possible to
satisfy~(\ref{twosided}) for any positive sequence~$(a_n)$ having the property 
that: $a_{n(j)+1}=O(a_{n(j)})$ as $j\to\infty$, for a sequence $n(j)$, $j\in \N$, such that
$a_{n(j)}\sim\sqrt{r_j R_j}$ as $j\to \infty$.

\section{Proof of Theorem~3}
\setcounter{equation}{0} To prove Theorem~\ref{main2} we need
several preliminary results. First we give a lemma based on two key ideas
from~\cite[proof of Theorem~1]{E}.
\begin{lemma}\label{Julesc}
Let $f$ be a {\tmf} and let $E\subset \C$ be a non-empty set.
\begin{itemize}
 \item[(a)] If $E$ has a subset $E'$ with at least 3~points, such
 that $E'$ is backwards invariant under~$f$, and int\,$E\cap J(f)=\emptyset$,
 then $J(f) \subset \partial E'$ and $J(f) \subset \partial E$.
 \item[(b)] If $z\in\partial E\setminus J(f)$, then $z$ lies in a
 Fatou component of $f$ which meets both $E$ and $E^c$; in particular,
 if every component of $F(f)$ that meets $E$ is contained in $E$,
 then $\partial E\subset J(f)$.
\end{itemize}
\end{lemma}
\begin{proof}
Since $E'$ is backwards invariant under~$f$ and $f$ is an open map,
we have $f(\C\setminus \overline{E'})\subset \C\setminus
\overline{E'}$. Thus $(f^n)$ forms a normal family in $\C\setminus
\overline{E'}$, by Montel's theorem, so $J(f) \subset
\overline{E'}\subset\overline{E}$. Since ${\rm int\,} E'\subset {\rm
int\,} E\subset F(f)$, we deduce that $J(f)\subset \partial E'$ and
$J(f)\subset \partial E$.

Part~(b) follows immediately from the fact that any open disc
centred at a point of $\partial E$ meets both $E$ and $E^c$.
\end{proof}
Proving that the hypotheses of~Lemma~\ref{Julesc}(a) hold for the
sets $L(f)$, $M(f)$ and $I^a(f)$ will be straightforward. However, to
apply Lemma~\ref{Julesc}(b) to these sets we must determine which Fatou components of
$f$ can meet them and also meet their complements. To do this, we use the
following distortion lemma, which is a combination
of~\cite[Lemma~7]{wB93} and~\cite[Theorem~1]{pR06}.
\begin{lemma}\label{distort}
Let $G$ be an unbounded open set in $\C$ such that $\partial G$
has at least two finite points, and let $f$ be analytic in $G$.
Let $D$ be a domain contained in $G$ such that $f^n(D)\subset G$,
for $n\ge 1$, and $f^n(z)\to\infty$ as $n\to\infty$, for $z\in D$.
\begin{itemize}
 \item[(a)] For any compact disc $\Delta\subset D$, there exist
$C>1$ and $n_0\in\N$ such that
\begin{equation}\label{weakdist}
 |f^n(z')|\le |f^n(z)|^C,\quad \text{for } z,z'\in \Delta,\;n\ge n_0.
 \end{equation}
 \item[(b)]If, in addition, one of the following holds:
\begin{itemize}
 \item[(i)] $\hat{\C}\setminus G$ contains an unbounded connected set, or
 \item[(ii)] $D=G$ and $f$ does not extend analytically to
 $\infty$,
\end{itemize}
then for any compact disc $\Delta\subset D$, there exist $C>1$ and
$n_0\in\N$ such that
\begin{equation}\label{strongdist}
 |f^n(z')|\le C|f^n(z)|,\quad \text{for } z,z'\in \Delta,\;n\ge n_0.
 \end{equation}
 \end{itemize}
\end{lemma}
We also need the following topological lemma. Here, for a set
$E\subset \C$, the set $\wt{E}$ denotes the union of $E$ and its
bounded complementary components.
\begin{lemma}\label{fill}
Let $V\subset \C$ be an open set with components $V_n$, $n\in
{\mathcal N}$, where $\mathcal N\subset \N$. For $n\in \mathcal N$,
put
\begin{equation}\label{omega}
\Omega_{V_n}=\bigcup \{\wt{V_m}: m\in\mathcal N, m\ne n,
\wt{V_n}\subset \wt{V_m}\}.
\end{equation}
Then the following cases can arise:
\begin{itemize}
 \item[(1)] $\Omega_{V_N}=\C=\wt{V_N}$ for some $N\in \mathcal N$;
 \item[(2)] $\Omega_{V_N}=\C$ for some $N\in \mathcal N$ but each
 $V_n$, $n\in\mathcal N$, is bounded, so
 there is a sequence $(n_j)$ in $\mathcal N$ such that
 \[ \wt{V_{n_1}}\subset \wt{V_{n_2}}\subset\cdots \quad\text{and}\quad \bigcup_{j \ge 1}
 \wt{V_{n_j}}=\C;\]
 \item[(3)] $\Omega_{V_N}$ is unbounded for some $N\in \mathcal N$ and $\Omega_{V_N}\ne
 \C$;
 \item[(4)] $\Omega_{V_n}$ is bounded for all $n\in \mathcal N$.
\end{itemize}
\end{lemma}
This lemma can be found in~\cite[Lemma~3]{pR92}. There the proof is
given for the case when $V=\{z:u(z)>M\}$, where $u$ is a real-valued,
continuous function defined in $\C$ and $M\in \R$, but this proof
applies without change when~$V$ is an arbitrary open set. Note that
each set $\Omega_{V_n}$ is {\it full\,}; that is, it has no bounded
complementary components.

Using Lemmas~\ref{distort} and~\ref{fill} we obtain the following
result about the behaviour of the iterates of a {\tmf} in its
escaping Fatou components, which may be of independent interest.
Later in the section, we give examples to show that
the condition that $U$ is not a {\pwd} is necessary in this result.
\begin{theorem}\label{Fatcomp}
Let $f$ be a {\tmf} and let $U$ be a component of $F(f)$ which is
contained in $I(f)$. If $U$ is not a {\pwd}, then the estimate~(\ref{strongdist})
holds for any compact disc $\Delta\subset U$.
\end{theorem}
\begin{proof}
For $n\in\N$, let $U_n$ be the Fatou component of $f$ such that
$f^n(U)\subset U_n$. By Lemma~\ref{fill}, applied to the open set
$V=\bigcup_{n\in\N}U_n$, the following cases can arise:
\begin{itemize}
 \item[(1)] $\Omega_{U_N}=\C=\wt{U_N}$ for some $N\in\N$;
 \item[(2)] $\Omega_{U_N}=\C$ for some $N\in\N$ but each $U_n$, $n\in\N$, is bounded, so
 there is a sequence $(n_j)$ in $\N$ such that
 \[ \wt{U_{n_1}}\subset \wt{U_{n_2}}\subset\cdots \quad\text{and}\quad \bigcup_{j \ge 1}
 \wt{U_{n_j}}=\C;\]
 \item[(3)] $\Omega_{U_N}$ is unbounded for some $N\in \N$ and $\Omega_{U_N}\ne
 \C$;
 \item[(4)] $\Omega_{U_n}$ is bounded for all $n\in \N$.
\end{itemize}
In case~(1), either $U_N$ is periodic, in which case $U_N$ is an invariant Baker
domain because $U_N\subset I(f)$, or all the Fatou components $U_n$, $n>N$ are
contained in bounded complementary components of $U_N$. If $U_N$ is an invariant
Baker domain, then Lemma~\ref{distort}(b)(ii) can be applied to $f$ in $G=U_N$
and hence~(\ref{strongdist}) holds for any compact disc $\Delta\subset
U$. On the other hand, if all the Fatou components $U_n$, $n>N$, are contained
in bounded complementary components of $U_N$, then Lemma~\ref{distort}(b)(i)
can be applied to $f$ in the open set $G=\bigcup_{ n>N}U_n$ and hence~(\ref{strongdist})
holds for any compact disc $\Delta\subset U$.

In case~(3), the boundary of $\Omega_{U_N}$ has an unbounded  component which
is contained in an unbounded complementary component of the set
\begin{equation}\label{G}
G=\bigcup_{n \in\N} U_n.
\end{equation}
Since $G$ is invariant under $f$, we can apply
Lemma~\ref{distort}(b)(i) to $f$ in $G$ with $D=U_1$ to deduce that~(\ref{strongdist})
holds for any compact disc $\Delta\subset U$.

In case~(4), the complement of the set
\[\bigcup_{n \in \N} \Omega_{U_n}\]
is connected and unbounded and lies in the complement of the
set $G$ defined in~(\ref{G}), so we can apply Lemma~\ref{distort}(b)(i) again 
to $f$ in $G$ with $D=U_1$.

Thus the estimate~(\ref{strongdist}) holds for any compact disc
$\Delta\subset U$ unless case~(2) holds. In this case $U$ is a
{\pwd}, so the result follows.
\end{proof}

We now use Theorem~\ref{Fatcomp} and Lemma~\ref{distort} to prove the
following result. This indicates when Lemma~\ref{Julesc}(b) can be applied 
if $E$ is $L(f)$, $M(f)$ or $I^a(f)$.
\begin{lemma}\label{lemma8b}
Let $f$ be a {\tmf}, let $U$ be a component of $F(f)$ and let
$a=(a_n)$ be a positive sequence with $a_n\to\infty$ as $n\to\infty$.
\begin{itemize}
 \item[(a)]If $U\cap L(f)\ne \emptyset$, then $U\subset L(f)$.
 \item[(b)]If $U\cap M(f)\ne \emptyset$, then $U\subset M(f)$.
 \item[(c)]If $U\cap I^a(f)\ne \emptyset$ and $U$ is not a {\pwd}, then $U\subset I^a(f)$.
\end{itemize}
\end{lemma}
\begin{proof}
First recall that if $U\cap I(f)\ne\emptyset$, then $U\subset I(f)$.

To prove part~(a), take $z\in U\cap L(f)$. Then
\[\limsup_{n\to\infty} \frac1n\log|f^n(z)|<\infty.\]
Let $\Delta$ be any compact disc in $U$ with centre~$z$. Then,
by~(\ref{weakdist}), there exists $C>1$ such that
\[\limsup_{n\to\infty} \frac1n\log|f^n(z')|\le
C\limsup_{n\to\infty} \frac 1n\log|f^n(z)|<\infty,\quad \text{for
} z'\in\Delta,\] so $\Delta\subset L(f)$. Hence $U\subset L(f)$. A
similar argument applies to $M(f)$.

To prove part~(c), take $z\in U\cap I^a(f)$. Then
\begin{equation}\label{Kest}
|f^n(z)|=O(a_n)\;\text{ as }n\to\infty.
\end{equation}
Since $U$ is not a {\pwd} we deduce by Theorem~\ref{Fatcomp} that if
$\Delta$ is any compact disc in $U$ with centre $z$, then there exist
$C>1$ and $n_0\in\N$ such that
\[|f^n(z')|\le C|f^n(z)|,\quad \text{for } z'\in\Delta,\; n\ge n_0,\]
so $\Delta\subset I^a(f)$ by~(\ref{Kest}). Hence $U\subset
I^a(f)$, as required.
\end{proof}

\begin{proof}[Proof of Theorem~\ref{main2}]
Let $a=(a_n)$ be a positive sequence such that $a_n\to \infty$ as
$n\to\infty$. We consider a positive sequence $a'=(a'_n)$ such that
\[a'_n\le
a_n\;\;\text{and}\;\; a'_n\;\text{is increasing}, \;\;\text{for }n\in
\N,\quad\text{and}\quad a'_n\to\infty\;\text{as }n\to\infty.\] It is
clear from the definitions that $I^{a'}(f)\subset I^a(f)$, that
$I^{a'}(f)$ is backwards invariant under $f$, and that each of the
sets $L(f)$ and $M(f)$ is completely invariant under $f$. Also, each
of the sets
\[
L(f)\cap J(f), \; M(f)\cap J(f)\;\;\text{and}  \;\;I^{a'}(f)\cap
J(f)
\]
is non-empty, by Theorem~\ref{main1}.

Therefore
\begin{itemize}
 \item each of $L(f)\cap J(f)$, $M(f)\cap J(f)$ and  $I^{a'}(f)\cap J(f)$
 is a non-empty backwards invariant subset of $I(f)\cap J(f)$
 and so is infinite, since for every $\zeta$ at least one of $\zeta, f(\zeta)$ or $f^2(\zeta)$ is
 not Fatou-exceptional and so has an infinite backwards orbit;
 \item each of $L(f)$, $M(f)$ and $I^a(f)$ contains no periodic points, and so has no interior points which
 lie in $J(f)$.
\end{itemize}
Thus, by Lemma~\ref{Julesc}(a), $L(f)$, $M(f)$ and $I^a(f)$ are each
dense in $J(f)$, and
\begin{equation}\label{Julsub}
J(f)\subset \partial L(f),\quad J(f)\subset \partial M(f) \;\;
\text{and}\;\; J(f)\subset \partial I^a(f).
\end{equation}

By Lemma~\ref{lemma8b}(a) and~(b), and Lemma~\ref{Julesc}(b), we
have
\[
\partial L(f)\subset J(f) \;\;\text{and}\;\;\partial M(f)\subset J(f),
\quad\text{so}\quad J(f)=\partial L(f)\;\;\text{and}\;\;
J(f)=\partial M(f).
\]
This proves part~(c).

Part~(d) follows immediately from~(\ref{Julsub}),
Lemma~\ref{lemma8b}(c) and Lemma~\ref{Julesc}(b). This completes
the proof of Theorem~\ref{main2}.
\end{proof}

We now show that we cannot strengthen the statement of
Theorem~\ref{main2}(d) or Corollary~\ref{cor1}(a) to assert that we
always have $J(f)=\partial I^a(f)$. More precisely, we show that for
a {\mf} $f$ with a finite number of poles, if $U$ is a {\bwd}, then
there exists a sequence $(a_n)$ for which $\partial I^a(f)\cap
U\ne\emptyset$ and hence $\partial I^a(f)\not\subset J(f)$. (Recall
that for such a {\mf}, any {\pwd} is a {\bwd}.) This theorem also 
shows that the conclusion of Theorem~\ref{Fatcomp} fails if $U$ is a {\bwd} of a {\tmffin}. 

\begin{theorem}\label{bwdexample}
Let $f$ be a {\tmffin} and with a {\bwd} $U$. Then there exist points $z_0,z'_0\in U$ such that
\[
\left|\frac{f^n(z'_0)}{f^n(z_0)}\right|\to\infty\;\;\text{as }n\to\infty.
\]
Thus, if $a_n=|f^n(z_0)|$, $n\in\N$, then $z_0\in I^a(f)$ but
$z'_0\notin I^a(f)$, so $\partial I^a(f)\cap U\ne \emptyset$.
\end{theorem}
\begin{proof}
First note that it is sufficient to construct the required points
$z_0$ and $z'_0$ in $f^N(U)$ where $N\ge 0$.

Since $f$ has a finite number of poles, it has a direct tract. Thus, by 
using a method of Eremenko (see~\cite[Theorem~1]{E}
and~\cite[proof of Theorem~3.1]{BRS08}), we can show that in any
annulus $A(\frac12 r,2r)$, where~$r$ is large enough, there exists
$z'_0$ such that
\begin{equation}\label{Epoint}
|f^{n+1}(z'_0)|\ge \frac12M(|f^{n}(z'_0)|,f),\quad\text{for
}n\ge 0.
\end{equation}
Next, by Lemma~\ref{Zheng}, there exists a sequence of
annuli $A(r_n,R_n)$, where $0<r_n<R_n$, such that
$r_n\to\infty$ as $n\to\infty$ and $R_n/r_n\to\infty$ as $n\to\infty$, and for $n$
large enough
\begin{equation}\label{annulus}
\overline{A(r_n,R_n)}\subset f^{n}(U).
\end{equation}
Thus we can choose $N\in\N$ so large that, with
$\rho_N=\sqrt{r_NR_N}$,
\[
\text{there exists }z'_0\in A(2\rho_N,R_N)\subset f^N(U) \text{ such
that~(\ref{Epoint}) holds, }
\]
\[
\text{(\ref{annulus}) holds for } n\ge N,
\]
and also, for all $r'>r\ge r_N$,
\begin{equation}\label{Mgrowth1}
M(r,f)>r
\end{equation}
and
\begin{equation}\label{Mgrowth2}
\frac{M(r',f)}{M(r,f)}\ge \left(\frac{r'}{r}\right)^2;\quad
\text{in particular,}\quad M(2r,f)\ge 4M(r,f).
\end{equation}
The estimate~(\ref{Mgrowth2}) is a special case of
Lemma~\ref{Had1}, or it follows directly from the convexity of
$\log M(r,f)$ with respect to~$\log r$.

Using~(\ref{Epoint}), the fact that $|z'_0|>2\rho_N$, and the second
estimate in~(\ref{Mgrowth2}), we deduce by induction that
\[|f^n(z'_0)|\ge 2M^n(\rho_N,f),\quad\text{for }n\in \N.\]
On the other hand, if $|z_0|=r_N$, then $z_0\in f^N(U)$ and
\[|f^n(z_0)|\le M^n(r_N,f),\quad\text{for }n\in \N.\]
Hence, by~(\ref{Mgrowth1}) and the first estimate
in~(\ref{Mgrowth2}), we deduce by induction that
\[
\frac{|f^n(z'_0)|}{|f^n(z_0)|}\ge
\frac{2M^n(\rho_N,f)}{M^n(r_N,f)}\ge
2\left(\frac{\rho_N}{r_N}\right)^{2^n}\to\infty\;\;\text{as
}n\to\infty,
\]
as required.
\end{proof}

In Theorem~\ref{bwdexample} the sequence $a_n=|f^n(z_0)|$ tends to
$\infty$ quickly; for example, $z_0\in Z(f)$. Our next example shows
that \begin{itemize}
 \item
 in the absence of a direct tract, $J(f)=\partial I^a(f)$ need not hold
 even if $a_n$ tends to $\infty$ at a much
 slower rate;
 \item  in Corollary~\ref{cor1}(a), the assumption about the existence of a direct tract cannot be
 omitted.
 \end{itemize}
This example also shows that the conclusion of Theorem~\ref{Fatcomp} can fail for a
{\pwd} $U$ that is not a {\bwd}.

\begin{example}\label{ex1}
There is a positive increasing sequence $a=(a_n)$ such that
$a_n\to\infty$ as $n\to\infty$ and $a_{n+1}=O(a_n)$ as $n\to\infty$,
and a {\tmf} $f$ with a Fatou component $U$ containing $z_0$ and
$z'_0$ such that
\begin{itemize}
 \item[(a)] $z_0\in I^a(f)$ but $z'_0\not\in I^a(f)$,
 so $\partial I^a(f)\cap U\ne \emptyset$;
 \item[(b)] $\limsup_{n\to\infty}\left|f^n(z'_0)/f^n(z_0)\right|=\infty$;
 \item[(c)] $U$ is a {\pwd} but not a {\bwd}.
\end{itemize}
\end{example}
\begin{proof}
We take $a,b,z_0$ and $z'_0$ such that
\begin{equation}\label{start}
1<a<z_0<z'_0<b,\quad a^{2}>4 b,\quad (z_0+1)^4> b^3
\quad\text{and}\quad z'_0>z_0+3.
\end{equation}
The sequence $(a_n)$ is defined as $a_n=b^{n+1}$, \;$n\ge 0$.

To construct $f$, we define a sequence of closed annuli $B_n$ as
follows. Here we use the notation
\[B(z;r,R)=\{w:r\le |w-z|\le R\},\quad\text{for } z\in\C,\;0<r<R.\]
First we choose a sequence $(n_k)$ in $\{0,1,2,\ldots\}$ such that
\begin{equation}\label{bineq}
b^{n_k}< (z_0+1)^{2^k}\le b^{n_k+1},\quad \text{for } k\ge 0.
\end{equation}
Then $n_0=0$, $n_1=1$ and $n_2= 3$, by~(\ref{start}). Also,
\begin{equation}\label{nprop}
n_{k+1}+1> 2n_k,\;\;\text{for }k\ge 0, \quad\text{so}\;\;
n_{k+1}\ge n_k+2,\;\;\text{for }k\ge 1,
\end{equation}
and we define the corresponding subsequence of closed annuli
\begin{equation}\label{Ank}
B_{n_k}=B(0;a^{2^k}, b^{2^k}),\quad k\ge 0.
\end{equation}

Now we define the rest of the $B_n$ to be certain nested finite
sequences of closed annuli lying between adjacent annuli
$B_{n_k}$. For $k\ge 1$, we put
\[p_k=\tfrac12 a^{2^k},\quad\text{so}\quad p_k>2b^{2^{k-1}},\]
by~(\ref{start}),
\[m_k=n_{k+1}-n_k-1,\;\;\text{for }k\ge 0,\;\;\text{so}\quad m_k\ge 1,\;\;\text{for }k\ge 1,\]
by~(\ref{nprop}), and choose $t_k$, $0<t_k<1$, such that
\[\frac{t_k}{a^{2^k}}<\frac{1} {2b^{2^k}}.\]
Then define
\[
B_{n_k+j}=B\left(p_k;\frac{t_k^{j-1}}{b^{2^k}},
\frac{t_k^{j-1}}{a^{2^k}}\right),\quad j=1,\ldots,m_k,
\]
which are $m_k$ nested disjoint closed annuli lying between
$B_{n_{k-1}}$ and $B_{n_{k}}$. The function
\begin{equation}\label{g1}
z\mapsto p_k+\frac{1}{z}, \quad z\in B_{n_k},\;k\ge 1,
\end{equation}
maps $B_{n_k}$ one-to-one onto $B_{n_k+1}$, and the function
\begin{equation}\label{g2}
z\mapsto p_k+t_k(z-p_k)
\end{equation}
maps $B_{n_k+j}$ one-to-one onto $B_{n_k+j+1}$, for
$j=1,\ldots,m_k-1$. Note that the annuli $B_{n_k}$, $B_{n_k+1}$,
\ldots, $B_{n_k+m_k}$ all have modulus $2^k \log(b/a)$.

Finally, for $k\ge 1$, we put $S_k=t_k^{2(m_k-1)}$, so that the
function
\begin{equation}\label{g3}
z\mapsto \frac{S_k}{(z-p_k)^2}
\end{equation}
maps $B_{n_k+m_k}$ two-to-one onto $B_{n_{k+1}}=B\left(0;a^{2^{k+1}},b^{2^{k+1}}\right)$.

Now we use~(\ref{g1}),~(\ref{g2}) and ~(\ref{g3}) to define a
function $g$ on the annuli $B_n$, $n\ge 1$, and we also set $g(z)=z^2$ on
$B_0$.

We have
\[
g^{n}(B_0)=B_n, \quad\text{for }n\ge 1,\quad\text{and}\quad
g^{n_k}(z)=z^{2^k},\quad\text{for }k\ge 0,
\]
since
\begin{equation}\label{zsquare}
g^{m_k+1}(z)=z^2,\quad\text{for }z\in B_{n_k},\;k\ge 0.
\end{equation}

Next, for $n\ge 0$, we choose closed annuli $B'_n\subset B_n$,
with $B'_n$ and $B_n$ concentric, and $\partial B'_n$ so close to
$\partial B_n$ that, for $n\ge 0$,
\begin{equation}\label{Bnclose}
z_0,z'_0\in B'_0,\quad g(B'_n)\subset B'_{n+1}\;\; \text{and}\;\; {\rm
dist}(\partial g(B'_n),\partial B'_{n+1})>0.
\end{equation}

Now, for $k\ge 0$, let
\[
E_k=\bigcup_{n=n_k}^{n_k+m_k}B_n,
\]
and let $P_k$ be a finite set, including $\infty$, with one point
in each component of $\hat{\C}\setminus (E_k\cup \{z:|z|\le
b^{2^{k-1}}\})$. Note that $E_k\cap\{z:|z|\le b^{2^{k-1}}\}=\emptyset$ and 
$g$ is analytic on each set $E_k$.

For each $k\ge 0$, we can use Runge's theorem (see~\cite{dG85}) to choose a rational
function $f_k$ with poles in the set $P_k$ such that
\begin{equation}\label{eps}
\sup_{E_k}|f_0+f_1+\cdots +f_k-g|<\eps_k\quad\text{and}\quad
\sup\{|f_k(z)|:|z|\le b^{2^{k-1}}\}<\eps_k,
\end{equation}
where the positive sequence $\eps_k$, $k\ge 0$, is so small that
\[
f(z)=f_0(z)+f_1(z)+\cdots\quad\text{is locally uniformly convergent on }\C,
\]
and (using~(\ref{Bnclose})) the function $f$ is so close to $g$ on
the annuli $B_n$, $n\ge 0$, that
\begin{equation}\label{Bninv}
f(B'_n)\subset B'_{n+1}, \;\;\text{for } n\ge 0,
\quad\text{so}\quad \bigcup_{n=0}^{\infty}B'_n\subset F(f)\cap
I(f),
\end{equation}
by Montel's theorem, and also, by~(\ref{zsquare}), that
\[
|f^{m_k+1}(z)-g^{m_k+1}(z)|=|f^{m_k+1}(z)-z^2|<1, \;\;\text{for }
z\in B'_{n_k},\;k\ge 0.
\]
Hence
\begin{equation}\label{forbit}
f^{n_k}(z)\in B'_{n_k}\quad\text{and}\quad
|f^{n_{k+1}}(z)-f^{n_k}(z)^2|<1, \;\;\text{for } z\in
B'_0,\;k\ge 0.
\end{equation}

By~(\ref{Bnclose}) and~(\ref{Bninv}) we see that $z_0$ and $z'_0$
lie in the same Fatou component of~$f$, say~$U$.
We now show that $z_0\in U\cap I^a(f)$, where $a_n=b^{n+1}$,\;$n\ge
0$. Let $X_k=|f^{n_k}(z_0)|$, \;$k\ge 0$. Then $X_0=z_0>1$
and, by~(\ref{Bnclose}) and~(\ref{forbit}),
\[X_{k+1}\le X_k^2+1,\quad\text{for }k\ge 0.\]
Thus by induction
\[X_k\le (X_0+1)^{2^k}-1,\quad\text{for }k\ge 0.\]
Therefore, for $k\ge 0$ and $j=0,1,\ldots, m_k$, by~(\ref{bineq}),
\[|f^{n_k+j}(z_0)|\le (X_0+1)^{2^k}-1\le b^{n_k+1}\le b^{n_k+j+1},\]
so $z_0\in I^a(f)$, as required.

Now let $Y_k=|f^{n_k}(z'_0)|$, \;$k\ge 0$. Then $Y_0=z'_0>1$ and,
by~(\ref{Bnclose}) and~(\ref{forbit}),
\[Y_{k+1}\ge Y_k^2-1,\quad\text{for }k\ge 0.\]
It again follows by induction that
\[Y_k\ge (Y_0-1)^{2^k}+1,\quad\text{for }k\ge 0.\]
Therefore, for $k\ge 0$, by~(\ref{start}) and~(\ref{bineq}),
\[
|f^{n_k}(z'_0)|\ge (z'_0-1)^{2^k}+1> (z_0+2)^{2^k}>
\left(\frac{z_0+2}{z_0+1}\right)^{2^k}b^{n_k}=
\left(\frac{z_0+2}{z_0+1}\right)^{2^k}\frac{a_{n_k}}{b},
\]
so $z'_0\notin I^a(f)$. This proves part~(a).

By Theorem~\ref{main2}(d), the Fatou component $U$ must be a
{\pwd}. Thus the Fatou components $U_n\supset B'_n$, $n\ge 0$, are
disjoint. Since $m_k\ge 2$ for $k\ge 3$, by~(\ref{nprop}), there
are infinitely many of these Fatou components that do not surround
$0$, so $U$ is not a {\bwd}.
\end{proof}
{\it Remark}\quad The proof of Example~\ref{ex1} can easily be
modified to give the same type of example in which $(a_n)$ is any
positive increasing sequence such that $a_n\to\infty$ as
$n\to\infty$ and $a_{n+1}=O(a_n)$ as $n\to\infty$.

\section{Proof of Theorem~\ref{main3}}
\setcounter{equation}{0} Theorem~\ref{main3} follows easily from
Theorem~\ref{main2}, Corollary~\ref{cor1} and the definition
of~$M(f)$.

\begin{proof}[Proof of Theorem~\ref{main3}]
In part~(a), $f$ has a direct tract and no {\bwd}s. Hence
\[J(f)=\partial L(f)=\partial M(f)=\partial I^a(f),\]
by Theorem~\ref{main2}(c) and Corollary~\ref{cor1}(a). In this
situation we know that $J(f)$ has at least one unbounded
component; see~\cite[Theorem~5.3]{BRS08}. Hence the sets $\partial
L(f)$, $\partial M(f)$ and $\partial I^a(f)$ each have at least
one unbounded component, as required.

To prove part~(a)(ii), recall that if $f$ is entire and $J(f)$ has
a bounded component, then $f$ has a {\bwd};
see~\cite[Theorem~1]{K98}. Thus all the components of $J(f)$ are
unbounded, so this is also true for all the components of
$\partial L(f)$, $\partial M(f)$ and $\partial I^a(f)$, as
required.

Theorem~\ref{main3}(b) follows immediately from the fact that if $f$
has a direct tract, then any {\bwd} lies in $Z(f)$, as mentioned in
the introduction before Corollary~\ref{cor1}, together with the fact
that $M(f)\cap Z(f)=\emptyset$.
\end{proof}

\section{Examples}
\setcounter{equation}{0} We end the paper by giving a number of
explicit examples to show how varied the structures of the sets
$L(f)$, $M(f)$ and $I^a(f)$ can be. Here, as usual, $a=(a_n)$ is a
positive sequence such that $a_n\to\infty$ as $n\to\infty$.

\begin{example}\label{ex2}
Let
\[
f(z)=\lambda e^z,\quad\text{where }0<\lambda<1/e.
\]
Then the components of $M(f)$ are all singletons, and hence so are
those of $L(f)$ and also $I^a(f)$ when $I^a(f)\subset M(f)$. However,
all the components of $\overline{L(f)}$, $\overline{M(f)}$ and
$\overline{I^a(f)}$ are unbounded.
\end{example}
\begin{proof}
In this case, $F(f)$ is a completely invariant immediate attracting basin and~$J(f)$ consists of uncountably many disjoint simple curves, each with one finite endpoint and the other endpoint at $\infty$; see~\cite{DT86}. The set $I(f)$ consists of the open curves
(without endpoints) together with some of their finite endpoints; see~\cite{bK99} and~\cite{lR06}. These open curves are in $Z(f)$, and even in $A(f)$; see~\cite{SZ} and~\cite{RS08c}. Thus $M(f)$ is contained in the set of 
finite endpoints, which is totally disconnected (see~\cite{M}), so the components of $M(f)$ are
singletons. However, the components of $\overline{L(f)}$,
$\overline{M(f)}$ and $\overline{I^a(f)}$ are all unbounded, by
Theorem~\ref{main3}(a)(ii).
\end{proof}

\begin{example}\label{ex3} Let \[f(z)=z+1+e^{-z}.\]
Then $f$ has a completely invariant Baker domain $U$ such that
$U\subset L(f)$. The sets $L(f)$, $M(f)$ and $I(f)$ are all connected
and dense in $\C$, as is $I^a(f)$ whenever $\liminf_{n\to\infty}
a_n/n>0$. However, all components of $M(f)\cap J(f)$ are singletons,
and hence so are those of $L(f)\cap J(f)$ and also $I^a(f)\cap J(f)$
when $I^a(f)\subset M(f)$.
\end{example}
\begin{proof}
In this case, $F(f)$ is a completely invariant Baker domain $U$
in which $f^n(z)\to \infty$ as $n\to\infty$ and $|f^n(z)|=O(n)$ as
$n\to\infty$; see~\cite[Example~1]{pF26}. Thus we have $U\subset
L(f)\subset M(f)$ and $U\subset I^a(f)$, whenever $\liminf_{n\to\infty}
a_n/n>0$. Note, however, that there are points of $\partial U$ (for example, fixed
points of $f$) which are not in $I(f)$. Each of the sets $L(f)$,
$M(f)$, $I^a(f)$ and $I(f)$ is connected and dense in $\C$, since
$U\subset I^a(f)\subset \overline{U}=\C$, for example.

It can be shown by using a result of Bara\' nski~\cite[Theorem~3]{kB08}, together with the fact that $f$ is the lift of $g(w)=(1/e)we^{-w}$ under $w=e^{-z}$, that $J(f)$ consists of uncountably many disjoint simple curves, each with one finite endpoint and the other endpoint at~$\infty$, and $I(f)\cap J(f)$ consists of the open curves together with some of their finite endpoints. Moreover, these open curves are in $A(f)$, so $M(f)\cap J(f)$ is contained in the set of finite endpoints, and its components are singletons; see~\cite{RS08c}.
\end{proof}

\begin{example}\label{ex4} Let \[f(z)=z+\sin z+2\pi.\]
Then every component of $F(f)$ is a bounded {\wand} whose closure is
contained in $L(f)$. The sets $I(f)$ and $L(f)$ are connected and
dense in~$\C$, and $I(f)\cap J(f)$ and $L(f)\cap J(f)$ are connected
and unbounded. Similar properties hold for $I^a(f)$ whenever
$\liminf_{n\to\infty} a_n/n>0$.
\end{example}
\begin{proof}
Since the function $f$ and $h(z)=z+\sin z$ are both lifts under
$w=e^{iz}$ of 
\[g(w)=w\exp(\tfrac12(w-1/w)),\quad z\in\C\setminus\{0\},\]
we have $F(g)=\exp(iF(f))=\exp(iF(h))$, by a result of Bergweiler~\cite{wB95a}. Now $F(g)$ consists of the basin of attraction of the super-attracting fixed point~$-1$, which is the only singular value
of~$g$, so $F(f)=F(h)$ is the lift of this basin. 

In particular, the point~$-1$ lifts to the points~$(2n+1)\pi,\;n\in\Z$, which are super-attracting fixed points of $h$. The immediate basin of attraction of~$\pi$ contains the
interval $(0,2\pi)$ and it is bounded, since $h$ maps the boundary of
the open rectangle $\{x+iy:0<x<2\pi, -3<y<3\}$ outside this rectangle; see~\cite{nF99} for a discussion of the mapping properties of such functions.

Thus $F(f)=F(h)$ consists of an infinite necklace of congruent
bounded Fatou components, say $U_n$, $n\in\Z$, where $(2n+1)\pi\in
U_n$, together with the successive preimages of these components
under $h$. The components $U_n$ form a wandering orbit under $f$ with
$f(U_n)=U_{n+1}$ for $n\in\Z$, in which $f^n(z)\to \infty$ as
$n\to\infty$ and $|f^n(z)|=O(n)$ as $n\to\infty$. Hence
\[\R\subset \bigcup_{n\in\Z}\overline{U_n}\subset L(f)\quad\text{and}\quad
\bigcup_{n\in\Z}\partial{U_n}\subset L(f)\cap J(f).\]

Now, the preimage under $f$ of the real axis consists of the real
axis itself and infinitely many pairs of curves tending to $\infty$
at both ends, each pair passing through a critical point $(2n+1)\pi$,
$n\in\Z$, and lying in $\{z:(2n+\tfrac12)\pi<\Re z
<(2n+\tfrac32)\pi\}$. Thus $F(f)$ consists of infinite necklaces of
bounded Fatou components together forming an infinite tree-like
structure. Therefore the set
\[E_1=\bigcup\{\overline{U}: U \;\text{is a Fatou component of}\; f\}\]
is connected and dense in~$\C$, and $E_1\subset L(f)\subset
I(f)\subset \C=\overline{E_1}$. Hence $I(f)$ and $L(f)$ are connected
and dense in~$\C$. The set
\[E_2=\bigcup\{\partial{U}: U \;\text{is a Fatou component of}\; f\}\]
is also connected and unbounded, and
\[E_2\subset L(f)\cap J(f)\subset I(f)\cap J(f)\subset J(f)
=\overline{E_1}\setminus F(f)=\overline{E_2}.\]
Hence $I(f)\cap J(f)$ and $L(f)\cap J(f)$ are connected and unbounded.
\end{proof}

\begin{example}\label{ex5} Let
\[f(z)=z+e^{-z}+2\pi i.\]
Then every component of $F(f)$ is an unbounded {\wand} $U$ contained
in $L(f)$ and also in $I^a(f)$ whenever $\liminf_{n\to\infty}
a_n/n>0$. However, the boundaries of these {\wand}s are not contained
in $L(f)$.
\end{example}
\begin{proof}
The function $h(z)=z+e^{-z}$ has congruent unbounded invariant Baker
domains $U_n$, $n\in \Z$, such that $2n\pi i\in U_n\subset
\{z:(2n-1)\pi<\Im(z)<(2n+1)\pi\}$, and $|h^m(z)|=O(m)$ as
$m\to\infty$, for $z\in U_0$; see~\cite{Pat2}. The Fatou set of $h$
consists of these Baker domains and their successive preimages under
$h$, which are all unbounded. Since $J(f)=J(h)$, by~\cite{wB95a}, the
components $U_n$ form a wandering orbit under $f$ with
$f(U_n)=U_{n+1}$ for $n\in\Z$, in which $f^n(z)\to \infty$ as
$n\to\infty$ and $|f^n(z)|=O(n)$ as $n\to\infty$. Hence
$F(f)=F(h)\subset L(f)$.

It was shown in~\cite[Theorem~6.1]{BD99} that $\Gamma_0=\{z:\Im
z=\pi\}\subset\partial U_0$ so $\Gamma_n=\{z:\Im
z=n\pi\}\subset\partial U_n$ for $n\in\Z$. It is easy to check
directly that each $\Gamma_n\subset Z(f)$, and even that
$\Gamma_n\subset A(f)$. Thus $\partial U_n$ is not contained in
$L(f)$, for $n\in\Z$, and the result follows.
\end{proof}

\begin{example}\label{ex6} Let
\[
f(z)=\frac12\left(\cos z^{1/4}+\cosh
z^{1/4}\right)=1+\frac{z}{4!}+\frac{z^2}{8!}+\cdots.
\]
Then $A(f)$ and $I(f)$ are connected, all components of $M(f)$ are bounded,
and all components of $\overline{L(f)}$, $\overline{M(f)}$ and
$\overline{I^a(f)}$ are unbounded.
\end{example}
\begin{proof}
The connectedness of $A(f)$ and $I(f)$ is proved in
~\cite[Corollary~5]{RS08b}. This proof depends on the fact that there
is a sequence of continua in $A(f)$ which surround $0$ and tend to
$\infty$, and this property forces all components of $M(f)$ to be
bounded. Also, $f$ has no {\bwd}s; see~\cite[Section~6]{RS08b}. Hence
all the components of $\overline{L(f)}$, $\overline{M(f)}$ and
$\overline{I^a(f)}$ are unbounded by Theorem~\ref{main3}(a)(ii).
\end{proof}
{\it Remark}\quad Using results in~\cite[see Theorem~2, its proof,
and Section~6]{RS08b}, we can show that there are many {\tef}s that
have the properties of Example~\ref{ex6}, as follows:
\begin{itemize}
 \item if $f$ is a {\tef} and there is a {\it hole} in $A(f)$, that is, a
bounded domain $G$ such that $\partial G\subset A(f)$ but $G\cap
J(f)\ne \emptyset$, then there is a sequence of continua in~$A(f)$
which surround~$0$ and tend to~$\infty$, from which it follows that
$A(f)$ and $I(f)$ are connected, and all the components of $M(f)$ are
bounded;
 \item there are many examples of {\tef}s for which there is a hole in
 $A(f)$ and no {\bwd}s -- all such functions must have the
 properties of Example~\ref{ex6}.
\end{itemize}
Note that in~\cite{RS08b} the set $A(f)$ is called $B(f)$ because an
alternative definition is used.

\begin{example}\label{ex7}
Let
\[f(z)=2z+2-\log 2-e^z.\]
Then $f$ has an invariant Baker domain $U$ such that $\overline{U}\setminus\{z_0\}\subset L(f)$,
where $z_0\in \partial U$ is a fixed point of $f$, and a bounded {\wand} $V$ such
that $\overline{V}\subset L(f)$.
\end{example}
\begin{proof}
It is shown in~\cite{wB95} that the function $f$ has
\begin{itemize}
\item an invariant Baker domain $U$ contained in $\{z:\Re z<0\}$
such that the map $f:U\to U$ is univalent and $\partial U$ is a
Jordan curve through $\infty$; \item a bounded Fatou component
$V_0$ containing the super-attracting fixed point $\log 2$; \item
bounded Fatou components of the form $V_k=\{z+2\pi ki:z\in V_0\}$,
$k\in \Z$, such that $f(V_k)=V_{2k}$, for $k\in\N$.
\end{itemize}
Thus $V=V_1$ is a bounded {\wand} such that $\overline{V}\subset L(f)$.

Also, it is easy to check that $\partial U$ meets the real axis at a repelling fixed point~$z_0$ of $f$ and that
\[(3/2)^n|z|\le |f^n(z)|\le 3^n |z|,\quad\text{for }z\in
\overline{U}\cap\{z:|z|\ge 2(3+\log 2)\},\] 
so
\begin{equation}\label{modlarge}
\overline{U}\cap\{z:|z|\ge 2(3+\log 2)\}\subset L(f).\end{equation}
Since $f$ is univalent on $U$,
it is conjugate, via a Riemann map, to a M\"obius transformation of the
unit disc onto itself. Since $\partial U$ is a Jordan curve the Riemann map extends to a homeomorphism on the closed unit disc, so the conjugate M\"obius transformation fixes two boundary points, one repelling and one attracting; the latter attracts all points of $\C$ except the repelling fixed point. It follows that
$\overline{U}\setminus\{z_0\}\subset I(f)$. Hence, by~(\ref{modlarge}), the whole of 
$\overline{U}\setminus\{z_0\}$ is contained in $L(f)$.

Note that in this case $J(f)$ is connected; see~\cite{K98}.
\end{proof}

Our final example shows that Theorem~\ref{main3}(a)(ii) is false without the assumption that
$f$ is entire.

\begin{example}\label{ex8}
Let \[f(z)=\lambda \sin z-\eps/(z-\pi),
\quad\text{where }0<\lambda<1,\; \eps>0 \text{ small}.\]
Then $f$ has a direct tract and no {\bwd}s, and $\overline{L(f)}$, $\overline{M(f)}$
and $\overline{I^a(f)}$ each have infinitely many bounded components.
\end{example}
\begin{proof}
The function $f$ has a direct tract, since it has only one pole, and
it has a completely invariant, unbounded, infinitely connected,
attracting Fatou component; see~\cite[Example~2]{Pat2}. Hence $J(f)$
has infinitely many bounded components and $f$ has no {\pwd}s. Thus
the sets $\overline{L(f)}$, $\overline{M(f)}$ and $\overline{I^a(f)}$
have infinitely many bounded components, by Theorem~\ref{main2}(c)
and (d).
\end{proof}


\begin{thebibliography}{99}



\bibitem{iB76} I.N. Baker, An entire function which has wandering domains,
{\it J. Austral. Math. Soc. Ser. A}, 22 (1976), 173--176.

\bibitem{iB84} I.N. Baker, Wandering domains in the iteration of entire functions,
{\it Proc. London Math. Soc.} (3), 49 (1984), 563--576.



\bibitem{BD99} I.N. Baker and P. Dom{\' i}nguez, Boundaries of unbounded
Fatou components of entire functions, {\it Ann. Acad. Sci. Fenn.
Math.\,}, 24 (1999), 437--464.




\bibitem{BL74} I.N.~Baker and L.S.O. Liverpool, Picard sets for entire functions, {\it Math. Z.,\,} 126 (1972), 230--238.

\bibitem{kB08} K. Bara\' nski, Trees and hairs for entire maps of finite order, {\it Math. Z.,\,} 257 (2007), no. 1, 33--59.

\bibitem{aB} A.F.~Beardon, Iteration of rational functions, Graduate Texts in Mathematics 132, Springer-Verlag, 1991.

\bibitem{wB93} W. Bergweiler, Iteration of meromorphic functions, {\it Bull. Amer. Math. Soc.}, 29 (1993), 151--188.

\bibitem{wB95} W. Bergweiler, Invariant domains and singularities,
{\it Math. Proc. Camb. Phil. Soc.\,}, 117 (1995), 525--532.

\bibitem{wB95a} W. Bergweiler, On the Julia set of analytic self-maps of the punctured plane, {\it Analysis}, 15 (1995), 251--256.

\bibitem{wB08} W. Bergweiler, An entire function with simply and multiply connected wandering domains,  to appear in {\it Pure Appl. Math. Quarterly}, arXiv:0708.0941.

\bibitem{BH99} W. Bergweiler and A. Hinkkanen, On semiconjugation of entire functions,
{\it Math. Proc. Camb. Phil. Soc.}, 126 (1999), 565--574.

\bibitem{BRS08} W. Bergweiler, P.J. Rippon and G.M. Stallard, Dynamics of
meromorphic functions with direct or logarithmic singularities,
{\it Proc. London Math. Soc.}, doi:10.1112/plms/pdn007.





\bibitem{DT86} R.L. Devaney and F. Tangerman, Dynamics of entire functions near the essential singularity, {\it Ergodic Theory Dynam. Systems }, 6 (1986), 498--503.



\bibitem{Pat2} P. Dom\'{\i}nguez, Dynamics of transcendental meromorphic
functions, {\it Ann. Acad. Sci. Fenn. Math. Ser. A} (1), 23 (1998), 225--250.

\bibitem{E} A.E. Eremenko, On the iteration of entire functions, {\it Dynamical systems and ergodic theory,} Banach Center Publications 23, Polish Scientific Publishers, Warsaw, 1989, 339--345.


\bibitem{nF99} N. Fagella, Dynamics of the complex standard family, {\it J. Math. Anal.
Appl.}, 229 (1999), 1--31.

\bibitem{pF26} P. Fatou, Sur l'it\'{e}ration des fonctions transcendantes enti\`{e}res,
{\it Acta Math.}, 47 (1926), 337--360.

\bibitem{dG85} D. Gaier, {\em Lectures on complex approximation}, Birkh{\"a}user, 1985.


\bibitem{wH} W.K. Hayman, {\it Meromorphic functions}, Clarendon Press, Oxford, 1964.

\bibitem{Hay} W.K. Hayman and P.B. Kennedy, {\it Subharmonic functions, Volume I}, Academic Press, 1976.



\bibitem{bK99} B. Karpi\'nska, Hausdorff dimension of the hairs without
endpoints for $\lambda \exp(z)$, {\it C. R. Acad. Sci. Paris S\'er. I Math.} 328 (1999), 1039--1044.

\bibitem{K98} M. Kisaka, On the connectivity of Julia sets of {\tef}s, {\it Ergodic Theory
Dynam. Systems}, 18 (1998), 189--205.


\bibitem{M} J.C. Mayer, An explosion point for the set of endpoints
of the Julia set of $\lambda\exp(z)$, {\it Ergodic Theory Dynam.
Systems}, 10 (1990), 177--183.









\bibitem{lR06} L.~Rempe, Topological dynamics of exponential maps on their escaping sets,
{\it Ergodic Theory Dynam. Systems}, 26 (2006), 1939--1975.


\bibitem{RS08c} L.~Rempe, P.J. Rippon and G.M. Stallard, Devaney hairs are fast escaping, Preprint.

\bibitem{pR92} P.J.~Rippon,  Asymptotic values of continuous
functions in Euclidean space, {\it Math. Proc. Camb. Phil. Soc.},
111 (1992), 309--318.


\bibitem{pR06} P.J. Rippon, Baker domains of meromorphic
functions, {\it Ergodic Theory Dynam. Systems}, 26 (2006),
1225--1233.


\bibitem{RS00} P.J. Rippon and G.M. Stallard, On sets where iterates of a
meromorphic function zip towards infinity, {\it Bull. London Math. Soc.},
32 (2000), 528--536.

\bibitem{RS05} P.J. Rippon and G.M. Stallard, On questions of Fatou and Eremenko,
{\it Proc. Amer. Math. Soc.}, 133 (2005), 1119--1126.

\bibitem{RS05a} P.J. Rippon and G.M. Stallard, Escaping points of meromorphic functions
with a finite number of poles, {\it Journal d'Analyse Math.}, 96 (2005), 225--245.

\bibitem{RS08a} P.J. Rippon and G.M. Stallard, On multiply
connected wandering domains of meromorphic functions, {\it J.
London Math. Soc.}, 77 (2008), 405-423.

\bibitem{RS08b} P.J. Rippon and G.M. Stallard, Escaping points of entire functions of small growth,
{\it Math. Z.}, doi:10.1007/s00209-008-0339-0.

\bibitem{RS08} P.J. Rippon and G.M. Stallard, Functions of small growth with no
unbounded Fatou components, To appear in {\it Journal d'Analyse
Math.}



\bibitem{SZ} D. Schleicher and J. Zimmer. Escaping points of
exponential maps. {\it J. London Math. Soc.} (2), 67 (2003), 380--400.






\bibitem{T} M. Tsuji, {\it Potential theory in modern function theory}, Maruzen, Tokyo, 1959;
reprint by Chelsea, New York, 1975.


\bibitem{jhZ02} J-H. Zheng, On uniformly perfect boundary of stable domains
in iteration of meromorphic functions II, {\it Math. Proc. Camb.
Phil. Soc.}, 132 (2002), 531--544.

\bibitem{jhZ06} J-H. Zheng, On multiply-connected Fatou components
in iteration of meromorphic functions, {\it J. Math. Anal. Appl.},
313 (2006), 24--37.

\end{thebibliography}
\end{document}